\documentclass[11pt,letterpaper]{amsart}

\usepackage{amssymb}
\usepackage{amsthm}
\usepackage{amsxtra}
\usepackage{mathrsfs}
\usepackage{graphicx}

\usepackage{afterpage}

\newcounter{myenum}

\newcommand{\R}{\operatorname{Re}}
\newcommand{\I}{\operatorname{Im}}
\newcommand{\ds}{\displaystyle}

\newtheorem{theorem}{Theorem}

\newtheorem{proposition}[theorem]{Proposition}
\newtheorem{lemma}[theorem]{Lemma}
\newtheorem{corollary}[theorem]{Corollary}

\newtheorem{claim}[theorem]{Claim}

\theoremstyle{remark}
\newtheorem{remark}[theorem]{Remark}

\newcommand{\cR}{\mathbb{R}}
\newcommand{\cS}{\mathcal{S}}
\newcommand{\cT}{\mathcal{T}}

\newcommand{\grad}{\nabla}

\DeclareMathOperator{\re}{Re} \DeclareMathOperator{\im}{Im}

\newcommand{\eps}{\varepsilon}

\numberwithin{equation}{section}

\numberwithin{theorem}{section}

\numberwithin{table}{section}

\numberwithin{figure}{section}

\title[Collapse and scattering for NLS]
{Going beyond the threshold:\\
scattering and blow-up
in the focusing NLS equation}

\author{Thomas Duyckaerts}
\address{Universit\'e Paris 13, France}
\author{Svetlana Roudenko}
\address{The George Washington University, USA}

\begin{document}

\begin{abstract}
We study the focusing nonlinear Schr\"odinger equation $i\partial_t u +\Delta u + |u|^{p-1}u=0$, $x\in \mathbb{R}^N$, in the $L^2$-supercritical regime with finite energy and finite variance initial data.
We investigate solutions above the energy (or mass-energy) threshold.
In our first result, we extend the known scattering versus blow-up dichotomy above that threshold for finite variance solutions in the energy-subcritical and energy-critical regimes, obtaining scattering and blow-up criteria for solutions with arbitrarily large mass and energy. As a consequence, we characterize the behavior of the ground state initial data modulated by a quadratic phase. Our second result gives two blow up criteria, which are also applicable in the energy-supercritical NLS setting. We finish with various examples illustrating our results.
\end{abstract}

\maketitle

\section{Introduction}

Consider the focusing nonlinear Schr\"odinger (NLS) equation on
$\mathbb{R}^N$:
\begin{equation}
 \label{E:NLS}
i\partial_t u +\Delta u + |u|^{p-1}u=0, \quad (x,t)\in \mathbb{R}^N
\times \mathbb{R},
\end{equation}
where $u=u(x,t)$ is complex-valued and the nonlinearity $p > 1+\frac4{N}$.
The solutions of this equation conserve mass, energy and momentum:
\begin{align*}
M[u](t) & = \int |u(x,t)|^2 \,dx = M[u](0),\\
E[u](t) & = \frac12\int |\nabla u(x,t)|^2 - \frac1{p+1}\int
|u(x,t)|^{p+1} \,dx = E[u](0),\\
P[u](t) & = \I \int \grad u(x,t) \, \bar{u}(x,t) \, dx = P[u](0).
\end{align*}

The equation \eqref{E:NLS} has scaling: $\ds u_\lambda(x,t) =
\lambda^{\frac2{p-1}} u(\lambda x, \lambda^2 t)$ is a solution if so
is  $u(x,t)$. This scaling produces a scale-invariant Sobolev norm $\dot H^{s_c}$ with
$$
s_c = \frac{N}2 - \frac2{p-1}.
$$
The nonlinearity restriction $p>1+\frac4{N}$ implies that we only consider the case $s_c > 0$.

\subsection{Scattering and blow up in the energy subcritical and critical cases}
For $p>1$, $N \geq 1$ such that $0\leq s_c<1$, we let $Q=Q_{p,N}$ be the unique $H^1$ radial positive solution of
\begin{equation}
\label{E:Q-groundstate}
\Delta Q - \left(1-s_c\right)\, Q + |Q|^{p-1}Q = 0.
\end{equation}
If $s_c=1$, i.e., $N\geq 3$ and $p=\frac{2N}{N-2}$, since the equation \eqref{E:Q-groundstate} is invariant by scaling, the radial positive solution equation \eqref{E:Q-groundstate} is no longer unique. In this case, we let
\begin{equation*}
Q_{\frac{2N}{N-2},N}=\frac{1}{\left(1+\frac{|x|^2}{N(N-2)}\right)^{\frac{N-2}{2}}},
\end{equation*}
which is often denoted by $W$, see \cite{KM}. In both cases, $Q_{p,N}$ is smooth. If $s_c<1$, $Q_{p,N}$ and all its derivatives decay exponentially at infinity. If $s_c=1$, $Q_{p,N}=W$ belongs to the homogeneous space $\dot{H}^1$. It is in $L^2$ if and only if $N\geq 5$. In all cases,
\begin{equation}
\label{E:uQ}
 u_Q(x,t)=e^{i(1-s_c)t}Q(x)
\end{equation}
is a solution of \eqref{E:NLS}. Let us emphasize that the choice of the constant $1-s_c$ in front of \eqref{E:Q-groundstate} is for convenience. If $s_c<1$, we can replace this constant by any positive constant by scaling. Similarly, if $s_c=1$, the choice $Q_{p,N}=W$ is arbitrary, and we could replace $W$ by $\lambda^{\frac{N}{2}-1}W(\lambda x)$ for any $\lambda>0$. We will state all our results using scale invariant quantities that do not depend on these choices.

One useful constant scaling quantity is $M[u]^{1-s_c} E[u]^{s_c}$, which we renormalize (for $s_c>0$) as
\begin{equation}
\label{E:ME}
\mathcal{ME} = \frac{M[u]^{\frac{1-s_c}{s_c}} E[u]}{M[Q]^{\frac{1-s_c}{s_c}} E[Q]}
\end{equation}
and call it the {\it mass-energy}. As it turns out, it is important to know its size relative to 1. We refer to $\mathcal{ME}=1$ as the {\it mass-energy threshold} (or the {\it energy threshold}, $\mathcal{E}=1$, when $s_c=1$).
The other useful scaling quantities (changing in time) are
$\|u\|^{1-s_c}_{L^2(\cR^N)} \|\nabla u(t)\|^{s_c}_{L^2(\cR^N)}$ and $\|u\|^{1-s_c}_{L^2(\cR^N)} \| u(t)\|^{\frac{p+1}{2}s_c}_{L^{p+1}(\cR^N)}$, for the purpose of this paper we use the last one.
\medskip

\underline{The case $0 <s_c<1$} (the mass-supercritical and energy-subcritical NLS), or
\begin{equation}
 \label{E:p-subcrit}
\frac{4}{N}+1<p<\infty ~\mbox{when}~ N=1,2 \quad \mbox{and} \quad \frac{4}{N}+1<p<\frac{4}{N-2}+1 ~\mbox{when} ~ N\geq 3.
\end{equation}
A physically important equation in this range ($s_c=\frac12$) is the 3d cubic NLS equation, for which the behavior of solutions was studied in series of papers \cite{CMP, DHR, AMRX, DR, CPDE}. It was later extended in \cite{CG} to the 2d quintic NLS (also $s_c=\frac12$) and then generalized to other dimension and nonlinearities ($0<s_c<1$) in \cite{G} (see also \cite{G-thesis}, and \cite{AN}, \cite{C+chinese}). When $\mathcal{ME} < 1$, the global behavior of solutions is completely understood, which we summarize in the following
\begin{theorem}
\label{T:DHR}
Let $u(x,t)$ be a solution of \eqref{E:NLS}, $0 < s_c <1$, with $u_0\in H^1(\cR^N)$. Assume $0<\mathcal{ME}<1$.
\begin{enumerate}
\item \label{I:Bup}
If $M[u_0]^{1-s_c}\left(\int |u_0|^{p+1}\right)^{s_c}< M[Q]^{1-s_c}\left(\int |Q|^{p+1}\right)^{s_c}$, then $u(t)$ exists globally  and, in fact, scatters in both time directions,  in $H^1$, to a linear solution.
\item \label{I:scatt}
If $M[u_0]^{1-s_c}\left(\int |u_0|^{p+1}\right)^{s_c}>M[Q]^{1-s_c}\left(\int |Q|^{p+1}\right)^{s_c}$, either $u(t)$ blows-up in finite positive time or there exists a sequence $t_n\nearrow +\infty$ such that $\lim_n\|\nabla u(t_n)\|_{L^2} = \infty$.  A similar statement holds for negative time. Furthermore, if $u_0$ has finite variance or $u_0$ is radial, then $u(t)$ blows-up in finite positive time and finite negative time.
\end{enumerate}
\end{theorem}
\begin{remark}
The above theorem is usually formulated with the gradient $\|\nabla u\|_{L^2}$ instead of the $\|u\|_{L^{p+1}}$ norm, we show the equivalence in Claim \ref{Cl:variational}.
\end{remark}
Behavior of solutions at the mass-energy threshold $\mathcal{ME}=1$ is completely classified in \cite{DR} in the case $N=3$, $p=3$, see Theorems 2 and 3 there.
\medskip

\underline{The case $s_c=1$} (the energy-critical NLS), or
\begin{equation}
 \label{E:p-crit}
p=\frac{4}{N-2}+1, \quad N \geq 3.
\end{equation}
In this case instead of $\mathcal{ME}$ we simply use the notation $\mathcal{E} = E[u]/E[W]$. In the case of $\mathcal{E} <1$ the behavior of solutions is also completely understood and is summarized in
\begin{theorem}
\label{T:KMD}
Let $s_c =1$ and $u(x,t)$ be a solution of \eqref{E:NLS} with $u_0\in \dot{H}^1(\cR^N)$. Assume $0<\mathcal{E} < 1$.
\begin{enumerate}
\item
If $\int |u_0|^{p+1} < \int |W|^{p+1}$ and $u$ is radial if $N=3,4$, then $u(t)$ exists globally and, in fact, scatters in $\dot{H}^1$ in both time directions.
\item
If $\int |u_0|^{p+1} > \int |W|^{p+1}$ and either $u_0$ is radial with $u_0 \in L^2$ or $x u_0 \in L^2$, then $u(t)$ blows-up in finite positive time and finite negative time.
\end{enumerate}
\end{theorem}
The above results in both cases $0<s_c \leq 1$ use the concentration compactness - rigidity method,
first introduced in the energy-critical case by Kenig-Merle \cite{KM}, where they proved Theorem \ref{T:KMD} in dimensions $N=3,4,5$. The higher dimensions extensions and non-radial assumption are in \cite{KiVi}.

Behavior of radial solutions at the energy threshold $\mathcal{E}=1$ is classified in \cite{DM}, see Theorem 2 there.
\medskip

Above the mass-energy threshold, i.e., $\mathcal{ME} >1$, the question about the global behavior of solutions is mostly open. For the radial 3d cubic NLS ($s_c<1$), in \cite{NS-2012} Nakanishi and Schlag described the global dynamics of $H^1$ solutions slightly above the mass-energy threshold, $\mathcal{ME}<1+\epsilon$. Beceanu in \cite{B-CPAM} constructs a co-dimension 1 manifold invariant by the flow of $\dot{H}^{1/2}$ solutions close to $u_Q$. The only other result which also works above the threshold is the two blow up criteria in \cite{HPR} (for the 3d cubic NLS).

In this paper we investigate solutions above this threshold, in particular, we improve the results of Theorems \ref{T:DHR} and \ref{T:KMD} for the finite variance solutions, where for globally existing solutions we also show scattering. Note that we can now describe solutions which are not necessarily $\epsilon$-close to the threshold.

Before we state the main results of the paper, we define the variance as
\begin{equation}
\label{E:variance}
V(t) = \int |x|^2 |u(x,t)|^2 \, dx.
\end{equation}
Assuming $V(0)<\infty$ (referred to as finite variance), the following virial identities hold:
\begin{align}
\label{V'}
V_{t}(t) &= 4 \I \int x\cdot \nabla u(x,t) \, \overline{u}(x,t) \, dx, \quad \mbox{and}\\
\label{V''}
V_{tt}(t)& =8\int |\nabla u(t)|^2-\frac{4N(p-1)}{p+1}\int |u(t)|^{p+1} \quad \\
\label{V''2}
\qquad & \equiv 4N(p-1)\, E[u] - 4(p-1)s_c \, \|\nabla u(t)\|_{L^2(\cR^N)}^2.
\end{align}

We abbreviate $Q = Q_{p,N}$ from \eqref{E:Q-groundstate}.
\begin{theorem}
\label{T:BB}
Let $u$ be a solution of  \eqref{E:NLS}, where $p$ satisfies \eqref{E:p-subcrit} or \eqref{E:p-crit}.
Assume $V(0)<\infty$, $u_0 \in H^1(\cR^N)$, and
\begin{equation}
\label{cond_1}
\mathcal{ME}[u] \, \left(1 - \frac{(V_t(0))^2}{32 \,E[u] \,V(0)} \right)\leq 1.
\end{equation}
\begin{enumerate}
\item[\underline{\bf Part 1}](Blow up)
If
\begin{equation}
\label{cond_Bup_2}
M[u_0]^{1-s_c}\left(\int |u_0|^{p+1}\right)^{s_c}>M[Q]^{1-s_c}\left(\int |Q|^{p+1}\right)^{s_c}
\end{equation}
and
\begin{equation}
\label{cond_Bup_1}
V_t(0) \leq 0,
\end{equation}
then $u(t)$ blows-up in finite positive time, $T_+(u)<\infty$.

\item[\underline{\bf Part 2}] (Boundedness and scattering)
If
\begin{equation}
\label{cond_bounded_2}
M[u_0]^{1-s_c} \left(\int |u_0|^{p+1}\right)^{s_c}<M[Q]^{1-s_c}\left(\int |Q|^{p+1}\right)^{s_c}
\end{equation}
and
\begin{equation}
\label{cond_bounded_1}
V_t(0) \geq 0,
\end{equation}
then
\begin{equation}
\label{good_bound}
\limsup_{t \to T_+(u)} ~~ M[u_0]^{1-s_c} \left(\int |u(t)|^{p+1}\right)^{s_c}<M[Q]^{1-s_c}\left(\int |Q|^{p+1}\right)^{s_c},
\end{equation}
in particular, in the energy-subcritical case when $p<\frac{N+2}{N-2}$, we get $T_+=+\infty$.

Furthermore, if $s_c<1$, $u$ scatters forward in time in $H^1$; if $s_c=1$, $u$ scatters forward in time in $\dot{H}^1$ provided $N\geq 5$ or $u$ is radial.
\end{enumerate}
\end{theorem}
\begin{remark}
\label{R:right-hand}
If $\mathcal{ME} < 1$, the conclusion of Theorem \ref{T:BB} follows from Theorems \ref{T:DHR} (if $s_c<1$) and \ref{T:KMD} ($s_c=1$). Theorem \ref{T:BB} is new only in the case when $\mathcal{ME} \geq  1$. \end{remark}
\begin{remark}
Let $\Sigma=\{f\in H^1,\; \int |x|^2f<\infty\}$.
The proof of Theorem 1.4 shows that the two subsets of $\Sigma$: $\Sigma_{Bup}$ defined by the conditions \eqref{cond_1}, \eqref{cond_Bup_2} and \eqref{cond_Bup_1}, and $\Sigma_{sc}$ defined by the conditions \eqref{cond_1}, \eqref{cond_bounded_2}, \eqref{cond_bounded_1} are stable by the forward flow of \eqref{E:NLS}. These two sets contain solutions with zero momentum and arbitrary large mass and energy (see Remark \ref{R:infinite_energy} below).
\end{remark}

\begin{remark}
We prove in Section \ref{S:Scattering} that any solution of \eqref{E:NLS} with property \eqref{good_bound} scatters for positive time (see Theorems \ref{T:crit_scatt1} and \ref{T:scatt_sub1}).
Note that if the $L^{p+1}$ norm is replaced by the gradient norm, the result is known, for example see \cite[Cor 5.16]{KM} in the energy-critical case. Our assumption \eqref{good_bound} is weaker, due to the one side implication in \eqref{claim1}, thus, Theorems \ref{T:crit_scatt1} and \ref{T:scatt_sub1} improve known results.
\end{remark}
\begin{remark}
The statement of Theorem \ref{T:BB} is not symmetric in time as the statements in Theorems \ref{T:DHR} and \ref{T:KMD}.
\end{remark}
\begin{remark}
The scattering statement (Part 2) of Theorem \ref{T:BB} is optimal in the following sense: if $u_0\in H^1$ has finite variance, and $u$ scatters forward in time, then there exists $t_0$ such that \eqref{cond_1}, \eqref{cond_bounded_2} and \eqref{cond_bounded_1} are satisfied by $u(t)$, $V(t)$ and $V_t(t)$ for all $t\geq t_0$. Indeed, if $u(t)$ scatters forward in times, then $E[u]>0$,
$$
 \|u(t)\|_{L^{p+1}}\rightarrow 0 \quad V(t)\sim 8E[u]t^2,\quad V_t(t)\sim 16 E[u]t, \quad \text{as }t\to+\infty, $$
which proves these three conditions.
\end{remark}

As a consequence of Theorem \ref{T:BB}, we obtain the behavior of solutions that are obtained by multiplying a finite-variance solutions with $\mathcal{ME}\leq 1$ by $e^{i\gamma |x|^2}$, $\gamma\in \cR$:
\begin{corollary}
\label{C:KeMe}
Let $\gamma\in \cR\setminus \{0\}$, $v_0\in H^1$ with finite variance be such that $\mathcal{ME}[v_0] \leq 1$, and $u^{\gamma}$ be the solution of \eqref{E:NLS} with initial data
$$
u_0^{\gamma}=e^{i\gamma |x|^2}v_0.
$$

If $M[u_0^{\gamma}]^{1-s_c} \left(\int |u_0^{\gamma}|^{p+1}\right)^{s_c} > M[Q]^{1-s_c}\left(\int |Q|^{p+1}\right)^{s_c}$,
then $\forall ~ \gamma<0$, $T_+(u^{\gamma})<\infty$.

If $M[u_0^{\gamma}]^{1-s_c} \left(\int |u_0^{\gamma}|^{p+1}\right)^{s_c} < M[Q]^{1-s_c}\left(\int |Q|^{p+1}\right)^{s_c}$, then for all $\gamma>0$, $u^{\gamma}$ satisfies \eqref{good_bound}. Furthermore, if $s_c<1$, $u^{\gamma}$ scatters forward in time in $H^1$. If $s_c=1$, $u^{\gamma}$ scatters forward in time in $\dot{H}^1$ provided $N\geq 5$, or $u$ is radial.
\end{corollary}
\begin{remark}
\label{R:infinite_energy}
The above corollary implies that we can predict the behavior of some solutions with arbitrary large energy: for example, if $v_0$ is such that $\mathcal{ME}[v_0] \leq 1$ and $\gamma > 0$ is large, then
$$
E[u_0^\gamma]=E[v_0] + 4\gamma^2 \|x v_0\|^2_{L^2} + 4\gamma \I \int x\cdot \nabla v_0 \bar{v}_0,
$$
and $E[u_0^\gamma] \nearrow \infty$ as $\gamma \to \pm\infty$.
Note that we can have $P[u_0^\gamma]=0$ for all $\gamma$ (this is the case for example if $v_0$ is radial): in particular, our results cannot be obtained from Theorem \ref{T:BB} by Galilean invariance as for example in \cite[Theorem 4]{DR}.
 \end{remark}

The second part of Corollary \ref{C:KeMe} is in accordance with the observation, made in \cite{CW}, that if $v_0\in H^1$ has finite variance, then the solution of \eqref{E:NLS} with initial data $e^{i\gamma|x|^2}v_0$ scatters forward in time for large, positive $\gamma$.
Let us also mention that in the mass-critical case $s_c=0$, the solution with initial data $e^{i\gamma |x|^2}v_0$ can be obtained explicitly, by the pseudo-conformal transformation from the solution with initial data $e^{i\gamma |x|^2}$. This transformation is not available if $s_c\neq 0$.

Another consequence of Theorem \ref{T:BB} is that we now understand the behavior of the ground state modulated by a quadratic phase in both time directions (which is important in studying blow up solutions, for example, see \cite{HPR-AMS}).
\begin{corollary}
\label{C:ground_states}
\underline{Subcritical case:}
Let $p$ be as in \eqref{E:p-subcrit} (i.e., $0 < s_c < 1$).
Let $\gamma \in \cR$ and $Q^{\gamma}$ be the solution of \eqref{E:NLS} with initial data
$$
Q^{\gamma}_0=e^{i\gamma|x|^2} \, Q(x) ,
$$
where $Q = Q_{p,N}$ is as in \eqref{E:Q-groundstate}.

If $\gamma>0$, then $Q^{\gamma}$ is globally defined, bounded and scatters forward in time and blows up backward in time. If $\gamma<0$, then $Q^{\gamma}$ blows up forward in time and is globally defined,  bounded and scatters backward in time.
\smallskip

\underline{Critical case:}
Let $p$ be as in \eqref{E:p-crit} (i.e., $s_c = 1$) with $N\geq 7$. Let $W^{\gamma}$ be the solution of \eqref{E:NLS} with initial data
$$
W^{\gamma}_0=e^{i\gamma|x|^2} W(x),
$$
where $W = Q_{p,N}$ as in \eqref{E:Q-groundstate} for $p, N$ such that $s_c=1$.

If $\gamma>0$, then $W^{\gamma}$ is globally defined, bounded and scatters forward in time and blows up backward in time. If $\gamma<0$, then $W^{\gamma}$ blows up forward in time and is globally defined, bounded and scatters backward in time.
\end{corollary}
\begin{remark}
In the case $p=3$, $N=3$, Nakanishi and Schlag has proved in \cite{NS-2012} the existence of an open subset of initial data such that the corresponding solutions scatters forward in time and blows up in finite negative time. Corollary \ref{C:ground_states} gives an explicit family of examples of such solutions for all mass-supercritical energy-subcritical nonlinearities. See also discussion after Conjecture 1 in \cite{HPR}, where such solutions (not necessarily close to $Q$) were exhibited.
\end{remark}

Another consequence of Theorem \ref{T:BB} is the behavior of the initial data with $V_t(0)=0$ (e.g., real-valued data) at the threshold $\mathcal{ME}=1$.
\begin{corollary}
\label{T:BB-realvalued}
Let $u(x,t)$ be a solution of \eqref{E:NLS}, $0 < s_c \leq 1$, with $V(0) < \infty$,
$V_t(0)=0$ and $u_0 \in H^1(\cR^N)$.
Assume $\mathcal{ME}=1$.
\begin{enumerate}
\item
If $M[u_0]^{1-s_c} \left(\int |u_0|^{p+1}\right)^{s_c} < M[Q]^{1-s_c}\left(\int |Q|^{p+1}\right)^{s_c}$,
then the solution $u(t)$ is bounded in $H^1$ ($\dot{H}^1$ if $s_c=1$). Moreover if  $s_c<1$, then $u$ is global and scatters in $H^1$ in both time directions; if $s_c=1$, then $u$ is global and scatters in $\dot{H}^1$ in both time directions, provided $u$ is radial in dimensions $N=3,4$.
\item
If $M[u_0]^{1-s_c} \left(\int |u_0|^{p+1}\right)^{s_c} > M[Q]^{1-s_c}\left(\int |Q|^{p+1}\right)^{s_c}$, then $u(t)$ blows-up in both time directions.
\end{enumerate}
\end{corollary}
Note that this result is a consequence of the classification of the solutions at the threshold in the energy-critical case \cite{DM} and in the 3d cubic case \cite{DR}.

\subsection{Blow up criteria in the mass-supercritical case}
We next consider any mass-supercritical NLS ($s_c>0$), including the energy-supercritical case:
\begin{equation}
 \label{E:p-supercrit}
p > \frac{4}{N-2}+1, \quad N \geq 3.
\end{equation}
There is not much known in this case. For the focusing NLS one has small data theory in the critical Sobolev space for global-in-time solutions and negative energy finite variance criteria for blow up in finite time solutions. In the defocusing case (when the sign in front of the nonlinearity is changed to minus), in \cite{KiVi2} it is shown that the {\it a priori} boundedness of solutions in the critical Sobolev norm implies scattering in high dimensions ($N \geq 5$), with additional technical assumptions on $p$, and numerical simulations in \cite{CSS} confirm boundedness of the corresponding invariant Sobolev norm ($H^2$ in that case) for the 5d quintic NLS equation ($s_c=2$). The motivation for these papers came from similar results in the energy-subcritical case (see \cite{KM3}) as well as results in the energy-supercritical regime for the nonlinear wave equation, initiated in \cite{KM2}  (see also \cite{DKM} and references therein). We refer to \cite{DS} for the description of a stable blow-up in this context.

The classical blow up criterion of Vlasov-Petrishev-Talanov \cite{VPT70}, Zakharov \cite{Z72}, Glassey \cite{G77} use the convexity argument on the variance $V(t)$ to show that finite variance, negative energy solutions break down in finite time. In \cite{Lu95}, the second time derivative of the variance is used as well, however, it is expressed in a dynamic way, which with a classical mechanics approach gives a more refined blow-up criterion. In \cite{HPR} that and another criteria were shown for the 3d cubic NLS equation; in particular, it was shown that there is an open set of blow up solutions above the mass-energy threshold $\mathcal{ME} > 1$. We extend this argument to any focusing mass-supercritical NLS equation in all dimensions and show that these conditions indeed produce new blow up solutions; for example, in the energy-critical case see \S \ref{S:energy-critical} and Figures \ref{F:s=1-N=3} and \ref{F:s=1-N=4}, and in the energy-supercritical case refer to \S \ref{S:energy-supercritical} and
Figure \ref{F:s>1-N=4}.

If $s_c>1$, equation \eqref{E:NLS} is not well-posed in $H^1$. To prove local well-posedness in the critical Sobolev space $\dot{H}^{s_c}$, one needs the nonlinearity to be at least $C^{s_c}$, i.e.
\begin{equation}
 \label{technical}
p\text{ is an odd integer or }N\leq 7\text{ or } p>\frac{N+2+\sqrt{N^2-4N-28}}{4}
\end{equation}
(note that the condition ``$N\leq 7\text{ or } p>\frac{N+2+\sqrt{N^2-4N-28}}{4}$'' is equivalent to $p>s_c$).
We abbreviate $M = M[u]$ and $E = E[u]$, and state the following two criteria:

\begin{theorem}
\label{T:1}
Suppose that $u_0\in H^1$ and $V(0)<\infty$. If $s_c>1$, assume furthermore \eqref{technical} and $u_0\in \dot{H}^{s_c}$.
The following is a sufficient condition for blow-up in finite time for \eqref{E:NLS} with $s_c>0$ and $E[u] >0$:
\begin{equation}
 \label{E:L}
\frac{V_t(0)}{M} < \sqrt{8 N s_c} ~ g \left(\frac{4}{N s_c}\, \frac{E V(0)}{M^2} \right),
\end{equation}
where
\begin{equation}
\label{E:g-def}
g(x) =
\begin{cases}
\sqrt{ \frac{1}{k x^k} + x - \left(1+\frac1{k}\right)} & \text{if }0<x\leq 1 \\
-\sqrt{\frac{1}{k x^k} + x - \left(1+\frac1{k}\right)} & \text{if }x \geq 1
\end{cases}
\quad \text{with} \quad k = \frac{(p-1)s_c}{2},
\end{equation}
and the function $g$ is graphed in Figure \ref{F:g-graph} for various values of $k$.
\end{theorem}

\begin{theorem}
\label{T:2}
Suppose that $u_0\in H^1$ and $V(0) 
<\infty$. If $s_c>1$, assume furthermore \eqref{technical} and $u_0\in \dot{H}^{s_c}$.
The following is a sufficient condition for blow-up in finite time for NLS \eqref{E:NLS} with $s_c>0$ and $E[u] >0$:
\begin{equation}
 \label{E:LA}
\frac{V_t(0)}{M} < \frac{4\sqrt 2 (M^{1-s_c}E^{s_c})^\frac1{N}}{C}
~ g \left(C^2 \frac{E^\frac4{N(p-1)} V(0)}{M^{1+\frac{2(p+1)}{N(p-1)}}} \right)\,,
\end{equation}
where
\begin{equation}
\label{E:defC}
C=\left( \frac{2(p+1)}{s_c(p-1)}
\left(C_{p,N}\right)^{\frac{N(p-1)}2 +(p+1)} \right)^{\frac2{N(p-1)}}
\end{equation}
and $C_{p,N}$ is a sharp constant in the interpolation inequality \eqref{E:interp1}, given by \eqref{E:C-anyN-anypSimple}, the function $g$ is defined in \eqref{E:g-def} and graphed in Figure \ref{F:g-graph}.
\end{theorem}
Let us emphasize that in both Theorems, the additional assumption in the supercritical case $s_c>1$ is only needed to ensure local well-posedness of the solution (see \cite[Theorem 3.1]{KiVi2}).

\begin{figure}
\includegraphics[scale=0.8]{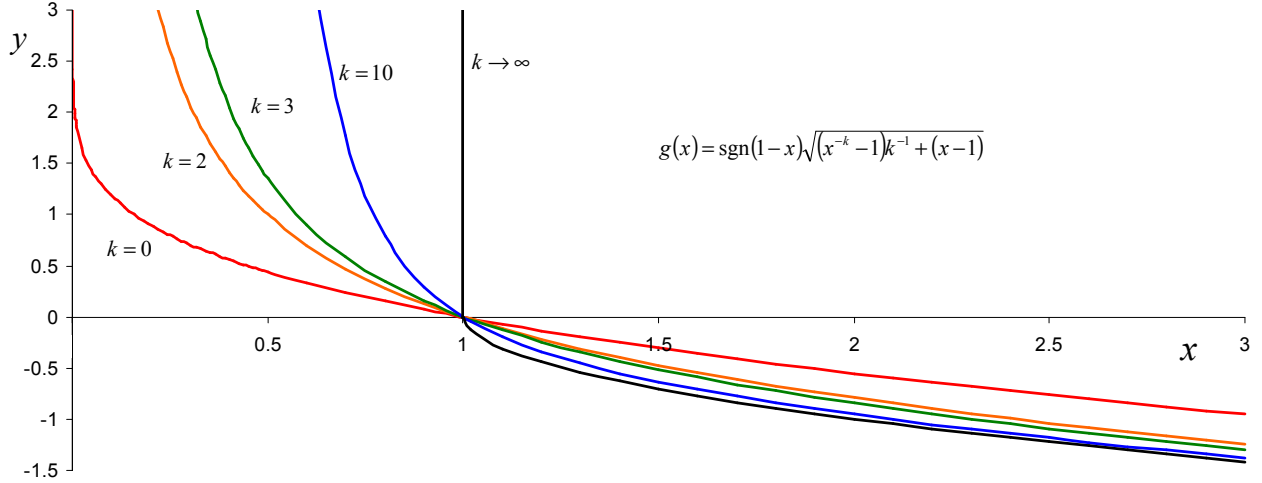}
\caption{A plot of the function $g(x)$ defined in \eqref{E:g-def} for various values of $k$, where $k = \frac{(p-1)s_c}{2}$. This function appears in Theorems \ref{T:1} and \ref{T:2}. The two limiting cases, $k \to 0$ (corresponding to $p \to 1+ \frac4{N}$) and $k \to \infty$ (corresponding to $p \to \infty$ or $N \to \infty$), are also indicated on the graph, for details refer to \eqref{E:functionflimits}.}
  \label{F:g-graph}
\end{figure}

Observe that both conditions deal with the normalized first derivative of the variance $\ds \frac{V_t(0)}{M}$ and the scaling-invariant quantities: $\ds \frac{V(0) E}{M^2}$ in Theorem \ref{T:1} and $\ds \frac{V(0) \,E^{\frac4{N(p-1)}}}{M^{1+\frac{2(p+1)}{N(p-1)}}}$ in Theorem \ref{T:2}.
For different values of $p$ and $N$, each criterion produces a different range of blow up solutions. For example, for the real-valued data that depends on the size of $M[u]^{1-s_c} E[u]^{s_c}$, see \eqref{E:real-comparison}.
A simplified version of Theorems \ref{T:1} and \ref{T:2} for real data is given in Section \S \ref{S:real}.
\bigskip

The structure of this paper is as follows: in Section \S \ref{S:Boundedness+Blowup}
we consider the energy-critical and energy-subcritical NLS equations and
prove the boundedness and blow up in finite time parts of Theorem \ref{T:BB}, then in Section \S \ref{S:Scattering} we show scattering for the bounded solutions (in the same range $0<s_c\leq 1$). In Section \S \ref{S:blowupcriteria} we investigate other blow up criteria, which are also valid for the energy-supercritical NLS equation. A sharp interpolation inequality is discussed in Section \S \ref{S:interpolation}, which is the key for Theorem \ref{T:2}. We conclude the paper with Section \S \ref{S:Examples}, where we illustrate Theorems \ref{T:1} and \ref{T:2} on the gaussian initial data in the energy-critical, supercritical and subcritical cases.

\subsection{Acknowledgements}
S.R. was partially supported by the NSF grants DMS-1103274 and CAREER-1151618.
T.D. was partially supported by ERC Grant Dispeq, ERC Avanced Grant  no. 291214, BLOWDISOL and ANR Grant SchEq.

\section{Boundedness and Blow-up in the case $0<s_c \leq 1$.} \label{S:Boundedness+Blowup}

We start with recalling the Gagliardo-Nirenberg inequality from \cite{W83}
which is valid for values $p$ and $N$ such that $0 \leq s_c \leq
1$ (when $s_c=1$ it is the critical Sobolev inequality):
\begin{equation}
 \label{E:GN}
\Vert f \Vert^{p+1}_{L^{p+1}(\cR^N)} \leq c_{GN} \, \Vert \nabla f \Vert_{L^2(\cR^N)}^{\frac{N(p-1)}{2}} \, \Vert f \Vert_{L^2(\cR^N)}^{2-\frac{(N-2)(p-1)}{2}},
\end{equation}
with equality when $f=Q$, where $Q$ is the ground state solution of \eqref{E:Q-groundstate}.
Rewriting \eqref{E:GN} as
\begin{equation}
\label{E:sharpGN}
\left(\int |f|^{p+1}\right)^{\frac{4}{N(p-1)}} \leq c_Q \, M[f]^{\kappa}\int |\nabla f|^2, \qquad \kappa=\frac{2(p+1)}{N(p-1)}-1,
\end{equation}
we have
\begin{equation}
\label{E:kappa_cQ}
\qquad c_Q = (c_{GN})^{\frac4{N(p-1)}}=\frac{\left(\int Q^{p+1}\right)^{\frac{4}{N(p-1)}}}{M[Q]^{\kappa}\int |\nabla Q|^2}. 
\end{equation}
Note that $\kappa>0$ if $0\leq s_c<1$ and $\kappa=0$ if $s_c=1$. Using the Pohozhaev identity:
\begin{equation}
 \label{E:Pohozaev}
 \int |\nabla Q|^2=\frac{N(p-1)}{2(p+1)} \int Q^{p+1},
\end{equation}
we get the following expressions for $c_Q$
\begin{equation}
 \label{E:cQ2}
c_Q =\frac{2(p+1)}{N(p-1)}\frac{\left(\int Q^{p+1}\right)^{\frac{4}{N(p-1)}-1}}{M[Q]^{\kappa}}
=\left( \frac{8(p+1)}{A} \right)^{\frac{4}{N(p-1)}}\frac{s_c}{N} \frac{(E[Q])^{\frac{4}{N(p-1)}-1}}{M[Q]^{\kappa}},
\end{equation}
where
\begin{equation}
\label{E:A-const}
A=2(N(p-1)-4).
\end{equation}


Our next observation is the following inequality, a consequence of \eqref{E:sharpGN} and Cauchy-Schwarz inequality in the spirit of Lemma 2.1, from the work of V.~Banica \cite{Banica}:
\begin{lemma}
\label{L:ValeriaGN}
Let $f\in H^1$ such that $xf\in L^2$. Then
$$
\left(\I\int x\cdot \nabla f\, \bar{f}\right)^2\leq \int |x|^2 |f|^2\left[\int |\nabla f|^2 -\frac{1}{c_QM[f]^{\kappa}} \left(\int |f|^{p+1}\right)^{\frac{4}{N(p-1)}}\right].
$$
\end{lemma}
\begin{proof}
The proof is similar to the one in \cite{Banica}. We provide it for the sake of completeness.
We apply \eqref{E:sharpGN} to $e^{i\lambda |x|^2} f$, $\lambda\in \cR$.
Using that
$$
\int \left|\nabla\big( e^{i\lambda |x|^2}f\big) \right|^2 =
4\lambda^2\int |x|^2|f|^2 + 4\lambda \I \int x\cdot \nabla f\,\bar{f} + \int |\nabla f|^2,
$$
and using the Gagliardo-Nirenberg inequality \eqref{E:sharpGN}, we get
$$
\forall \lambda\in \cR,\quad c_Q M[f]^{\kappa}\left[ 4\lambda^2 \int |x|^2|f|^2 + 4\lambda\I \int x\cdot\nabla f\,\overline{f}
+ \int |\nabla f|^2 \right]-\left(\int |f|^{p+1}\right)^{\frac{4}{N(p-1)}} \geq 0,
$$
where the left-hand side is a polynomial in $\lambda$.
The discriminant of this polynomial in $\lambda$ must be negative, which yields the conclusion of the Lemma.
\end{proof}
\begin{remark}
 \label{R:equal_ValeriaGN}
Assume $f=e^{i\lambda |x|^2}Q$ for some $\lambda\in \cR$. Then the polynomial in the proof of Lemma \ref{L:ValeriaGN} admits $-\lambda$ as a double root, and its discriminant is $0$. As a consequence, the inequality in the conclusion of Lemma \ref{L:ValeriaGN} is an equality. Combining with \eqref{E:cQ2}, we get
$$
f=e^{i\lambda|x|^2}Q\Longrightarrow \left(\I \int x\cdot \nabla f \overline{f}\right)^2= \int |x|^2|f|^2\left(\int |\nabla f|^2-\frac{N(p-1)}{2(p+1)}\int |f|^{p+1}\right).
$$
\end{remark}

We next show a variational result which is a consequence of Gagliardo-Nirenberg (or Sobolev) inequality \eqref{E:sharpGN}.
\begin{claim}
\label{Cl:variational}
Let $f$ be in ${H}^1$ ($f$ in $\dot{H}^1$ if $s_c=1$). Then
\begin{multline}
\label{claim1}
\left(\int |\nabla f|^2\right)^{s_c}M[f]^{1-s_c}<\left(\int |\nabla Q|^2\right)^{s_c}M[Q]^{1-s_c}\\
\Longrightarrow
\left(\int |f|^{p+1}\right)^{s_c}M[f]^{1-s_c}<\left(\int |Q|^{p+1}\right)^{s_c}M[Q]^{1-s_c}.
\end{multline}
Assume furthermore that
\begin{equation}
 \label{sub_energy2}
M[f]^{\frac{1-s_c}{s_c}}E[f] \leq M[Q]^{\frac{1-s_c}{s_c}}E[Q], \quad \mbox{or} \quad \mathcal{ME} \leq 1.
\end{equation}
Then the reverse implication to \eqref{claim1} holds, and we obtain
\begin{multline}
\label{claim2}
\left(\int |\nabla f|^2\right)^{s_c}M[f]^{1-s_c}<\left(\int |\nabla Q|^2\right)^{s_c}M[Q]^{1-s_c}\\
\iff
\left(\int |f|^{p+1}\right)^{s_c}M[f]^{1-s_c}<\left(\int |Q|^{p+1}\right)^{s_c}M[Q]^{1-s_c}.
\end{multline}
Moreover, \eqref{claim2} also holds with non-strict inequalities (in the case of equality, $f$ is equal to $Q$ up to space translation, scaling and phase.)
\end{claim}
\begin{proof}
Using the inequality \eqref{E:sharpGN} with the value of $c_Q$ as in \eqref{E:kappa_cQ}, we write it in the renormalized form:
\begin{equation}
 \label{sharpGN2}
\left(\frac{M[f]^{\frac{1-s_c}{s_c}} \int |f|^{p+1}}{M[Q]^{\frac{1-s_c}{s_c}}\int |Q|^{p+1}}\right)^{\frac{4}{N(p-1)}} \leq
\frac{M[f]^{\frac{1-s_c}{s_c}}\int |\nabla f|^2}{M[Q]^{\frac{1-s_c}{s_c}}\int |\nabla Q|^2}.
\end{equation}
The implication \eqref{claim1} follows immediately.

Assume \eqref{sub_energy2}. In view of \eqref{claim1}, we only have to show the implication from right to left in \eqref{claim2}. Assume
\begin{equation*}
\left(\int |\nabla f|^{2}\right)^{s_c}M[f]^{1-s_c}>\left(\int |\nabla Q|^{2}\right)^{s_c}M[Q]^{1-s_c} .
\end{equation*}
Then
$$
E[Q]M[Q]^{\frac{1-s_c}{s_c}}\geq E[f]M[f]^{\frac{1-s_c}{s_c}}>\frac 12 \int |\nabla Q|^2M[Q]^{\frac{1-s_c}{s_c}}-\frac{1}{p+1}\int |f|^{p+1}M[f]^{\frac{1-s_c}{s_c}},
$$
and the desired inequality
\begin{equation*}
\left(\int |f|^{p+1}\right)^{s_c}M[f]^{1-s_c}>\left(\int |Q|^{p+1}\right)^{s_c}M[Q]^{1-s_c}
\end{equation*}
follows.
\end{proof}

\subsection{Proof of Theorem \ref{T:BB}}
In this part we prove Theorem \ref{T:BB}, except for the scattering statement in the end of this theorem which is proved in Section \ref{S:Scattering}.

The conclusion of Theorem \ref{T:BB} is known if $\mathcal{ME} \leq 1$ (see Remark \ref{R:right-hand}).
We will thus assume
\begin{equation}
\label{cond_1_r}
\mathcal{ME} > 1.
\end{equation}

Recalling the variance $V(t)$ from \eqref{E:variance} and its second derivative \eqref{V''2}, we obtain
\begin{equation}
 \label{norms_u}
\int |u|^{p+1}=\frac{(p+1)(16E[u]-V_{tt})}{2A} \, ,\quad  \int |\nabla u|^2=\frac{4N(p-1)E[u]-V_{tt}}{A},
\end{equation}
where $A$ is defined in \eqref{E:A-const}.
Note that the first expression in \eqref{norms_u} implies that $V_{tt}\leq 16E[u]$ for all $t$.

Using the definition of $V_t$ from \eqref{V'} and Lemma \ref{L:ValeriaGN}, we get
\begin{equation}
 \label{bound_V'}
(V_t(t))^2 \leq 16 V(t) \left[\int |\nabla u(t)|^2-\frac{1}{c_Q M[u]^{\kappa}} \left(\int |u(t)|^{p+1}\right)^{\frac{4}{N(p-1)}}\right].
\end{equation}

Substituting \eqref{norms_u} into \eqref{bound_V'} and abbreviating $E=E[u]$, we obtain
\begin{equation}
\label{bound_z'}
(z_t)^2\leq 4 \, \varphi(V_{tt}),
\end{equation}
where
$$
z(t)=\sqrt{V(t)}
$$
and
\begin{equation}
 \label{defphi}
\varphi(\sigma) = -\frac{\sigma}{A}+\frac{4N(p-1)E}{A}-\frac{1}{c_Q M[u]^{\kappa}} \left(\frac{(p+1)(16E-\sigma)}{2A}\right)^{\frac{4}{N(p-1)}}
\end{equation}
is defined for $\sigma\in (-\infty,16E]$. We have
$$
\varphi'(\sigma)=-\frac{1}{A}+\frac{4}{c_Q M[u]^{\kappa} N(p-1)}\left(\frac{p+1}{2A} \right)^{\frac{4}{N(p-1)}} \left(16E-\sigma\right)^{\frac{4}{N(p-1)}-1}.
$$
Since $\frac{4}{N(p-1)}-1<0$ ($s_c>0$), $\varphi$ is decreasing on $(-\infty,\sigma_{m})$, increasing on $(\sigma_m,16E]$, where $\sigma_m$ is given by the equation
\begin{equation}
 \label{eq_sm}
\frac{1}{A}=\frac{4}{c_Q M[u]^{\kappa}N(p-1)}\left(\frac{p+1}{2A}\right)^{\frac{4}{N(p-1)}}
\left(16E-\sigma_m\right)^{\frac{4}{N(p-1)}-1}.
\end{equation}
Note that this implies that
\begin{equation}
 \label{phi_sm}
\varphi(\sigma_m)=\frac{\sigma_m}{8}.
\end{equation}
Furthermore, using \eqref{E:cQ2}, we can rewrite \eqref{eq_sm} as
\begin{equation}
 \label{good_sm}
\left(\frac{M[u]}{M[Q]}\right)^{1-s_c}\left(\frac{E[u]-\frac{1}{16}\sigma_m}{E[Q]}\right)^{s_c}=1.
\end{equation}
As a consequence, \eqref{cond_1_r} is equivalent to
\begin{equation}
 \label{cond_1_bis}
\sigma_m\geq 0,
\end{equation}
and \eqref{cond_1} is equivalent to
\begin{equation}
 \label{cond_1_ter}
z_t(0)^2\geq 4\varphi(\sigma_m)=\frac{\sigma_m}{2}.
\end{equation}

\noindent\underline{First case:}
we assume \eqref{cond_Bup_1} and \eqref{cond_Bup_2}. Note that \eqref{cond_Bup_1} means exactly
\begin{equation}
 \label{cond_Bup_1_bis}
z_t(0)\leq 0.
\end{equation}
In view of \eqref{E:cQ2}, the assumption \eqref{cond_Bup_2} is equivalent to
$$
\left(\frac{M[u]}{M[Q]}\right)^{1-s_c}\left(\frac{\frac{ A}{8(p+1)}\int |u_0|^{p+1}}{E[Q]}\right)^{s_c}>1=\left(\frac{M[u]}{M[Q]}\right)^{1-s_c} \left(\frac{E[u]-\frac{\sigma_m}{16}}{E[Q]}\right)^{s_c},
$$
that is, by \eqref{norms_u},
\begin{equation}
 \label{cond_Bup_2_bis}
V_{tt}(0)<\sigma_m.
\end{equation}
We will show by contradiction that
\begin{equation}
 \label{z''negative}
\forall t\in [0,T_+(u)), \quad z_{tt}(t)<0.
\end{equation}
Note that
\begin{equation}
\label{E:z-tt}
z_{tt}=\frac{1}{z} \left(\frac{V_{tt}}{2}-(z_t)^2\right),
\end{equation}
and that $z_{tt}$ is continuous on $[0,T_+(u))$.
By \eqref{cond_1_ter} and \eqref{cond_Bup_2_bis},
\begin{equation*}
z_{tt}(0)<0.
\end{equation*}
Assume that \eqref{z''negative} does not hold. Then there exists $t_0\in (0,T_+(u))$ such that
$$
\forall t\in [0,t_0), \; z_{tt}(t)<0\text{ and }z_{tt}(t_0)=0.
$$
By \eqref{cond_1_ter} and \eqref{cond_Bup_1_bis},
\begin{equation}
 \label{z'neg}
\forall t\in (0,t_0],\quad z_t(t)<z_t(0)\leq -\sqrt{4\varphi(\sigma_m)}.
\end{equation}
Hence,
$\left(z_{t}\right)^2>4\varphi(\sigma_m)$, which, combined with \eqref{bound_z'}, implies that
\begin{equation*}
 \forall t\in (0,t_0],\quad \varphi(V_{tt})>\varphi(\sigma_m).
\end{equation*}
As a consequence, $V_{tt}(t)\neq \sigma_m$ for $t\in (0,t_0]$, and by \eqref{cond_Bup_1_bis} and continuity of $V_{tt}$,
\begin{equation}
\label{V''sigma}
\forall t\in [0,t_0],\quad V_{tt}(t)<\sigma_m.
\end{equation}
Combining \eqref{z'neg} and \eqref{V''sigma}, we obtain
$$
z_{tt}(t_0)=\frac{1}{z(t_0)}\left(\frac{V_{tt}(t_0)}{2}-(z_t(t_0))^2 \right)<\frac{1}{z(t_0)}\left(\frac{\sigma_m}{2}-\frac{\sigma_m}{2}\right)=0,
$$
contrary to the definition of $t_0$. Thus, the proof of \eqref{z''negative} is complete.

Assume that $T_+(u)=+\infty$. Then by \eqref{cond_Bup_1_bis} and \eqref{z''negative},
$$
\forall t\geq 1,\quad z_t(t)\leq z_t(1)<0,
$$
a contradiction with the fact that $z(t)$ is positive.
\smallskip

\noindent\underline{Second case:}
we now assume, in addition to \eqref{cond_1} and \eqref{cond_1_r}, that \eqref{cond_bounded_1} and \eqref{cond_bounded_2} hold. In other words, in addition to \eqref{cond_1_bis} and \eqref{cond_1_ter}, we also assume the following inequalities
\begin{gather}
 \label{cond_bounded_1_bis}
z_t(0)\geq 0\\
\label{cond_bounded_2_bis}
V_{tt}(0)>\sigma_m.
\end{gather}
We first notice that there exists $t_0 \geq 0$ such that
\begin{equation}
\label{eq:1}
z_t(t_0) > 2 \sqrt{ \varphi(\sigma_m)}.
\end{equation}
Indeed, by \eqref{cond_1_ter} and \eqref{cond_bounded_1_bis}, $z_t(0) \geq 2 \sqrt{\varphi(\sigma_m)}$. If the inequality is strict, then we are done with $t_0=0$.
If not, then by \eqref{E:z-tt} and \eqref{cond_bounded_2_bis}, $z_{tt}(0)>0$ and \eqref{eq:1} follows for small $t_0>0$.

Let $\epsilon_0 >0$ be a small parameter and assume
\begin{equation}
\label{eq:2}
z_t(t_0) \geq 2\sqrt{\varphi(\sigma_m)} + 2 \epsilon_0.
\end{equation}
We will prove by contradiction
\begin{equation}
\label{eq:3}
\forall t \geq t_0 \quad z_t(t) > 2 \sqrt{\varphi(\sigma_m)} + \epsilon_0.
\end{equation}
Assume that \eqref{eq:3} does not hold, and let
\begin{equation}
\label{eq:4}
t_1 = \inf\{t \geq t_0 : z_t(t) \leq 2 \sqrt{\varphi(\sigma_m)}+\epsilon_0 \}.
\end{equation}

By \eqref{eq:2} $t_1>t_0$. By continuity
\begin{equation}
\label{eq:5}
z_t(t_1) = 2 \sqrt{\varphi(\sigma_m)} + \epsilon_0
\end{equation}
and
\begin{equation}
\label{eq:6}
\forall t \in [t_0, t_1] \quad z_t(t) \geq 2\sqrt{\varphi(\sigma_m)} + \epsilon_0.
\end{equation}
By \eqref{bound_z'}
\begin{equation}
\label{eq:6'}
\forall t \in [t_0,t_1] \quad (2\sqrt{\varphi(\sigma_m)}+\epsilon_0)^2\leq (z_t(t))^2 \leq 4 \varphi(V_{tt}(t)).
\end{equation}

As a consequence, $\varphi(V_{tt}(t)) > \varphi(\sigma_m)$ for all $t \in [t_0, t_1]$, thus,
$V_{tt}(t) \neq \sigma_m$ and by continuity $V_{tt}(t) > \sigma_m$ for $t \in [t_0, t_1]$.

We prove that there exists a universal constant $D>0$ such that
\begin{equation}
\label{eq:7}
\forall t \in [t_0,t_1] \quad V_{tt}(t) \geq \sigma_m + \frac{\sqrt{\epsilon_0}}{D}.
\end{equation}
Indeed, by the Taylor expansion of $\varphi$ around $\sigma=\sigma_m$, there exists $a>0$ such that
\begin{equation}
\label{eq:8}
|\sigma-\sigma_m| \leq 1 \quad \Rightarrow \quad \varphi(\sigma) \leq \varphi(\sigma_m) + a(\sigma-\sigma_m)^2.
\end{equation}
If $V_{tt}(t)\geq \sigma_m+1$, then \eqref{eq:7} holds (taking $D$ large).
If $\sigma_m<V_{tt}(t) \leq \sigma+1$, then by \eqref{eq:6'} and \eqref{eq:8}, we obtain
$$
(2\sqrt{\varphi(\sigma_m)} + \epsilon_0)^2 \leq (z_t(t))^2 \leq 4\, \varphi(V_{tt}(t)) \leq 4\,\varphi(\sigma_m) + 4 a (V_{tt}(t)-\sigma_m)^2,
$$
thus
$$
4\sqrt{\varphi(\sigma_m)}\epsilon_0 + \epsilon_0^2 \leq 4a(V_{tt}-\sigma_m)^2,
$$
and we get \eqref{eq:7} with $D = \sqrt{a}\,(\varphi(\sigma_m))^{-1/4}$.

However, by \eqref{E:z-tt} and \eqref{eq:5} we have
\begin{align*}
z_{tt}(t_1) &= \frac1{z(t_1)}\left(\frac{V_{tt}(t_1)}{2} - (z_t(t_1))^2 \right) \\
& \geq \frac1{z(t_1)}\left(\frac{\sigma_m}{2} + \frac{\sqrt{\epsilon_0}}{2D} - (2\sqrt{\varphi(\sigma_m)} + \epsilon_0)^2\right)\\
& \geq \frac1{z(t_1)}\left(\frac{\sqrt{\epsilon_0}}{2D} - 4\epsilon_0 \sqrt{\varphi(\sigma_m)} - \epsilon_0^2\right) > 0,
\end{align*}
if $\epsilon_0$ is small enough, thus, contradicting \eqref{eq:5} and \eqref{eq:6}. Therefore, we obtain \eqref{eq:3}. Note that we have also shown that the inequality \eqref{eq:7} holds for all $t \in [t_0,T_+(u))$. Hence (using the first equality in \eqref{norms_u}, Pohozhaev identity \eqref{E:Pohozaev} and the characterization \eqref{good_sm} of $\sigma_m$),
\begin{multline*}
M[u]^{1-s_c}\left(\int |u(t)|^{p+1}\right)^{s_c}\leq M[u]^{1-s_c}\left(\frac{p+1}{2A}(16E-\sigma_m-\sqrt{\eps_0}{M})\right)^{s_c}<M[u]^{1-s_c}\left(\frac{p+1}{2A}(16E-\sigma_m)\right)^{s_c}\\
=M[Q]^{1-s_c}\left(\int |Q|^{p+1}\right)^{s_c},
\end{multline*}
which gives \eqref{good_bound}. This concludes the proof of Theorem \ref{T:BB}, except for the fact that \eqref{good_bound} implies that the solution $u$ scatters forward in time, which is proved in Section \ref{S:Scattering}.

\subsection{Dichotomy for quadratic phase initial data}
We next study the behavior of solutions with data modulated by a quadratic phase, proving Corollary \ref{C:KeMe} except for the scattering statement which will follow from \eqref{good_bound'} and Section \ref{S:Scattering}:
\begin{corollary}
\label{C:KeMe'}
Let $\gamma\in \cR\setminus \{0\}$, $v_0$ be such that $\mathcal{ME}[v_0] \leq 1$, and $u^{\gamma}$ be the solution of \eqref{E:NLS} with initial data
$$
u_0^{\gamma}=e^{i\gamma |x|^2}v_0.
$$
\begin{itemize}
\item
If $\left(\int |v_0|^{p+1}\right)^{s_c} M[v_0]^{1-s_c} > \left(\int | Q|^{p+1}\right)^{s_c}M[Q]^{1-s_c}$, then
$$
\forall ~ \gamma<0,\quad T_+(u^{\gamma})<\infty.
$$
\item
If $\left(\int |v_0|^{p+1}\right)^{s_c}M[v_0]^{1-s_c} < \left(\int |Q|^{p+1}\right)^{s_c}M[Q]^{1-s_c}$, then
\begin{equation}
\label{good_bound'}
\forall ~ \gamma >0,\quad \limsup_{t \to T_+(u)} ~~ M[u_0]^{1-s_c} \left(\int |u^{\gamma}(t)|^{p+1}\right)^{s_c}<M[Q]^{1-s_c}\left(\int |Q|^{p+1}\right)^{s_c}.
\end{equation}
\end{itemize}
\end{corollary}
\begin{proof}
Let $v_0$ satisfy $\mathcal{ME}[v_0] \leq 1$, $\gamma \in \cR\setminus \{0\}$ and $u$ be the solution with initial data $u_0=e^{i\gamma|x|^2}v_0$ (we drop the superscripts $\gamma$ to simplify the notation).
If $\mathcal{ME}[u_0] \leq 1$,
then \eqref{claim2} in Claim \ref{Cl:variational} and the usual blow-up/scattering dichotomy
implies the result (see \cite{KM} or Theorem \ref{T:KMD}   for the energy-critical case, \cite{AMRX}, \cite{G} or Theorem \ref{T:DHR} for a general energy-subcritical case).
We thus assume
\begin{equation}
\label{super_energy}
\mathcal{ME}[u_0] \geq 1, \quad \mbox{or}, \quad
E[u_0]M[u_0]^{\frac{1-s_c}{s_c}}\geq E[Q]M[Q]^{\frac{1-s_c}{s_c}}.
\end{equation}
We will show that $u_0$ satisfies the assumptions of Theorem \ref{T:BB}. We have
\begin{gather}
\label{E_u0}
E[u_0]=E[v_0]+2\gamma \im \int x\cdot \nabla v_0\, \overline{v}_0 +2\gamma^2 \int |x|^2|v_0|^2 \quad \mbox{and}\\
\label{mom_u0}
\im \int \overline{u}_0 \, x \cdot \nabla u_0=
\im \int \overline{v}_0 \, x \cdot \nabla v_0+2\gamma \int |x|^2|v_0|^2.
\end{gather}
As a consequence,
\begin{equation}
\label{E:the_same}
E[u_0]-\frac{\left(\im \int x\cdot \nabla u_0\, \overline{u}_0\right)^2}{2\int |x|^2|u_0|^2}
= E[v_0]-\frac{\left(\im \int x\cdot \nabla v_0\, \overline{v}_0\right)^2}{2\int |x|^2|v_0|^2},
\end{equation}
and the assumption \eqref{cond_1} follows from writing out explicitly $\mathcal{ME}[v_0] \leq 1$.

We will only treat the case when
\begin{equation}
 \label{scatt_case}
\gamma>0 \quad\text{ and }\quad
M[v_0]^{\frac{1-s_c}{s_c}}\int |v_0|^{p+1}<M[Q]^{\frac{1-s_c}{s_c}}\int |Q|^{p+1},
\end{equation}
the proof of the other case is similar and is left to the reader. Of course,
\begin{equation}
 \label{hyp1OK}
M[u_0]^{\frac{1-s_c}{s_c}}\int |u_0|^{p+1} = M[v_0]^{\frac{1-s_c}{s_c}}\int |v_0|^{p+1}<M[Q]^{\frac{1-s_c}{s_c}} \int |Q|^{p+1},
\end{equation}
which shows that \eqref{cond_bounded_2} is satisfied. Since $\gamma$ is positive, we see by \eqref{E_u0} that \eqref{super_energy} implies that $\gamma\geq \gamma_c^+$, where $\gamma_c^+$ is the unique positive solution of
$$
\left(E[v_0]+2\gamma_c^+ \im \int x\cdot \nabla v_0 \overline{v}_0 +2(\gamma_c^+)^2 \int |x|^2|v_0|^2\right)M[v_0]^{\frac{1-s_c}{s_c}}=E[Q]M[Q]^{\frac{1-s_c}{s_c}}.
$$
Since $\mathcal{ME}[v_0]<1$, or equivalently, $M[v_0]^{\frac{1-s_c}{s_c}}E[v_0]\leq M[Q]^{\frac{1-s_c}{s_c}}E[Q]$, the above line implies
$$
\im \int x\cdot \nabla v_0\overline{v}_0+\gamma_c^+ \int |x|^2|v_0|^2\geq 0.
$$
Using that $\gamma\geq \gamma_c^+$, we see that
$$
\im \int x\cdot \nabla u_0\overline{u}_0 =
\im \int x\cdot \nabla v_0\overline{v}_0+2\gamma \int |x|^2|v_0|^2 > \gamma\int |x|^2|v_0|^2>0,
$$
which yields the assumption \eqref{cond_bounded_1}.
Theorem \ref{T:BB} applies, which concludes the proof of Corollary \ref{C:KeMe}.
\end{proof}

We now consider the ground state with the quadratic phase and prove Corollary \ref{C:ground_states}.

\begin{proof}[Proof of Corollary \ref{C:ground_states}]
Denoting $Q=W$, the proof is the same in the energy-critical case as in the energy-subcritical case and we shall not distinguish the two cases. Note that $xW\in L^2(\cR^N)$ if and only if $N\geq 7$, hence our assumption on the dimension in the energy-critical case.

Using that if $u(x,t)$ is a solution, then $\overline{u}(x,-t)$ is also a solution, it is sufficient to prove the assertions on $Q^{\gamma}$ for positive times. Assume that $\gamma$ is positive. Then $Q^{\gamma}$ almost satisfies the assumptions of Theorem \ref{T:BB}, in the sense that it satisfies \eqref{cond_1}, \eqref{cond_bounded_1} and the equality corresponding to the strict inequality in \eqref{cond_bounded_2}. We will show that the solution $Q^{\gamma}_{t_0}(x,t)=Q^{\gamma}(t_0+t,x)$ satisfies the assumptions \eqref{cond_1}, \eqref{cond_bounded_1} and \eqref{cond_bounded_2} for small positive $t_0$, which will imply by Theorem \ref{T:BB} that $Q^{\gamma}$ is bounded for positive time $t>0$.

We first note that
$$
\im \int x\cdot \nabla Q^{\gamma}_0 \, \overline{Q^{\gamma}_0}=2\gamma \int |x|^2|Q|^2,
$$
so that
\begin{equation}
\label{sign_V'Q}
\im \int x\cdot \nabla Q^{\gamma}_{t_0} \, \overline{Q^{\gamma}_{t_0}}>0
\end{equation}
for small $t_0$, which shows that $Q_{t_0}^{\gamma}$ satisfies \eqref{cond_bounded_1} for small $t_0$.

Now using that $Q^{\gamma}$ satisfies \eqref{E:NLS}, we get
\begin{equation*}
\frac{d}{dt} \int |Q^{\gamma}|^{p+1}=(p+1)\re \int \left|Q^{\gamma}\right|^{p-1} \overline{Q^{\gamma}}\,\partial_t Q^{\gamma}=-(p+1)\im \int |Q^{\gamma}|^{p-1}\overline{Q^{\gamma}} \Delta Q^{\gamma}.
\end{equation*}
Since, at $t=0$,
$$
\Delta Q_0^{\gamma}=e^{i\gamma|x|^2} \left(2N i\gamma Q + 4i\gamma \, x \cdot \nabla Q-4\gamma^2|x|^2 Q + \Delta Q\right),
$$
we get
\begin{equation*}
\left[\frac{d}{dt} \int |Q^{\gamma}|^{p+1}\right]_{\restriction t=0}=-2 \gamma N (p+1) \int Q^{p+1}-4\gamma (p+1)\int  Q^{p} x\cdot \nabla Q=-2\gamma N(p-1)\int Q^{p+1}<0.
\end{equation*}
As
\begin{equation*}
M\left[Q^{\gamma}_0\right]^{\frac{1-s_c}{s_c}}\int \left|Q^{\gamma}_0\right|^{p+1}=M[Q]^{\frac{1-s_c}{s_c}}\int |Q|^{p+1},
\end{equation*}
we obtain that $Q^{\gamma}_{t_0}$ satisfies assumption \eqref{cond_bounded_2} for small $t_0$. It remains to check \eqref{cond_1}.

Let
$$
F(t)=M[Q^{\gamma}]^{\frac{1-s_c}{s_c}}\left(E[Q^{\gamma}]-\frac{\left(\I \int x\cdot\nabla Q^{\gamma}(t) \, \overline{Q^{\gamma}}(t)\right)^2}{2\int |x|^2 |Q^{\gamma}(t)|^2} \right)-M[Q]^{\frac{1-s_c}{s_c}}E[Q].
$$
By \eqref{E:the_same} with $v_0=Q$, $F(0)=0$. We must check that $F(t)\leq 0$ for small positive $t$. We will use the same notations as in the proof of Theorem \ref{T:BB}:
$$
V(t)=\int |x|^2|Q^{\gamma}(x,t)|^2 \, dx,\quad z(t)=\sqrt{V(t)}.
$$
Then
$$
F(t)=M[Q^{\gamma}]^{\frac{1-s_c}{s_c}}\left(E[Q^{\gamma}]-\frac 18(z_t(t)^2)\right)-M[Q]^{\frac{1-s_c}{s_c}}E[Q].
$$
Thus,
$$
F_t(t)=-\frac 14 M[Q^{\gamma}]^{\frac{1-s_c}{s_c}}z_t(t)z_{tt}(t),
$$
and
$$
F_{tt}(0)=-\frac 14 M[Q^{\gamma}]^{\frac{1-s_c}{s_c}}\Big(z_t(0)z_{ttt}(0)+(z_{tt}(0))^2\Big).
$$
By Remark \ref{R:equal_ValeriaGN} and \eqref{V''}, $z_{tt}(0)=0$, and thus, $F_t(0)=0$ and
$$
F_{tt}(0)=-\frac 14 M[Q^{\gamma}]^{\frac{1-s_c}{s_c}}z_t(0)z_{ttt}(0).
$$
Using that
$$
V_{tt}=2(z_t)^2+2zz_{tt},\quad V_{ttt}=6z_tz_{tt}+2z\,z_{ttt},
$$
we obtain that $V_{ttt}(0)=2z(0)\,z_{ttt}(0)$, and (since, by \eqref{sign_V'Q}, $z_t(0)>0$) the sign of $F_{tt}(0)$ and $-V_{ttt}(0)$ is the same. By \eqref{norms_u}, we get that this sign is the same as the one of
$$
\left[\frac{d}{dt} \int |Q^{\gamma}|^{p+1}\right]_{\restriction t=0}.
$$
Hence, $F_{tt}(0)<0$, which shows that $F(t)$ is negative for small $t\neq 0$, thus, completing the proof.

If $\gamma<0$, one shows by a very close proof to the above that $Q^{\gamma}(t_0+t,x)$ satisfies the assumptions \eqref{cond_1}, \eqref{cond_Bup_1} and \eqref{cond_Bup_2} for small positive $t_0$, implying the blow-up result and concluding the proof of Corollary \ref{C:ground_states}.

\end{proof}

\section{Scattering} \label{S:Scattering}
In this section, we show that the bound from above \eqref{good_bound}, obtained in the previous section for the boundedness part of Theorem \ref{T:BB}, implies scattering of the solution.
Subsection \ref{SS:crit} is devoted to the energy-critical case, and subsection \ref{SS:subcrit} to the energy-subcritical case. Proofs rely on a compactness-rigidity argument of the type initiated in \cite{KM}. A refinement of this argument is necessary since smallness of the $L^{p+1}$ norm of the initial data does not insure global well-posedness and scattering of the corresponding solution.

\subsection{Energy-critical case}
\label{SS:crit}
Recall the NLS equation \eqref{E:NLS} when $s_c=1$ or \eqref{E:p-crit}, i.e., in dimension $N\geq 3$ we have
\begin{equation}
 \label{NLScrit}
i\partial_t u+\Delta u+|u|^{\frac{4}{N-2}}u=0,\quad u_{\restriction t=0}=u_0\in \dot{H}^1(\cR^N).
\end{equation}
In this part we show the scattering result of Theorem \ref{T:BB}, namely,
\begin{theorem}
\label{T:crit_scatt1}
Let $u$ be a solution of \eqref{NLScrit} with maximal time of existence $T_+(u)$, and assume
\begin{equation}
\label{limsup}
\limsup_{t\to T_+(u)} \int_{\cR^N} |u(x,t)|^{\frac{2N}{N-2}}\,dx<\int |W(x)|^{\frac{2N}{N-2}}\,dx.
\end{equation}
Assume furthermore that $u$ is radial if $N=3,4$.
Then $T_+(u)=+\infty$ and $u$ scatters forward in time.
\end{theorem}
If $I$ is a real interval, we define
$$
S(I)=L^{\frac{2(N+2)}{N-2}}(I\times \cR^N),
$$
noting that the pair $(\tfrac{2(N+2)}{N-2}, \frac{2(N+2)}{N-2})$ is $\dot{H}^1$-admissible.
Recall (see e.g. Cazenave's book \cite{bookC})  that if $u$ is a solution of \eqref{NLScrit} such that $\|u\|_{S(0,T_+(u))}<\infty$, then $T_+(u)=+\infty$ and $u$ scatters forward in time.

If $A>0$, $E_0\in \cR$, we let $\cS(E_0,A)$ be the supremum of $\|u\|_{S(I)}$, where $I$ is a real interval, and $u$ a solution of \eqref{NLScrit} on $I\times \cR^N$ such that
\begin{gather}
\label{defS1} E[u]\leq E_0\\
\label{defS2} \sup_{t\in I} \int |u(x,t)|^{\frac{2N}{N-2}}\leq A\\
\label{defS3} \text{ if }N=3,4, \; u \text{ is radial}.
\end{gather}
We deduce Theorem \ref{T:crit_scatt1} from a slightly stronger result:
\begin{theorem}
 \label{T:crit_scatt2}
Assume that $0<A<\int |W|^{\frac{2N}{N-2}}$ and $E_0\in \cR$. Then $\cS(E_0,A)$ is finite.
\end{theorem}
Of course, Theorem \ref{T:crit_scatt2} implies Theorem \ref{T:crit_scatt1}.

Theorem \ref{T:crit_scatt1} is a variant of the scattering part of the main Theorem of \cite{KM} (see also Corollary 5.18). We also refer to Theorem 1.7 of \cite{KiVi}, which states that if $A<\int |\nabla W|^2$ and $\cT(A)=\sup \|u\|_{S(I)}$, where the supremum is taken over all solutions on $I\times \cR^N$ such that $\sup_{t\in I} \int |\nabla u(t)|^2\leq A$, then $\cT(A)<\infty$. Note that by the critical Sobolev embedding,
$$
\int|\nabla u|^2<\int |\nabla W|^2\Longrightarrow \int |u|^{\frac{2N}{N-2}}< \int |W|^{\frac{2N}{N-2}},
$$
which shows that Theorem \ref{T:crit_scatt2} is slightly stronger than Theorem 1.7 of \cite{KiVi}.

The proof of Theorem \ref{T:crit_scatt2} follows the general strategy initiated in \cite{KM}, and is very close to the proof of \cite{KM}, with the extra argument given in \cite{KiVi} to deal with nonradial solutions in dimension $N\geq 5$. We only sketch the proof, highlighting the differences. We start by a purely variational result:
\begin{claim}
 \label{C:variational}
Let $a, A\in \cR$ such that $0<a<A<\int |W|^{\frac{2N}{N-2}}$. Then there exists $\eps_0=\eps_0(a,A)>0$  such that for all $f\in \dot{H}^1(\cR^N)$ with
$$
a \leq \int |f|^{\frac{2N}{N-2}}\leq A
$$
one has
\begin{align}
 \label{coerc}
\int |\nabla f|^2-\int |f|^{\frac{2N}{N-2}}\geq \eps_0\\
\label{positive_E}
E[f]\geq  \frac{\eps_0}{2}.
\end{align}
\end{claim}
\begin{proof}
By Sobolev inequality
$$
\int |\nabla f|^2-\int |f|^{\frac{2N}{N-2}}\geq \left(\int |f|^{\frac{2N}{N-2}}\right)^{\frac{N-2}{N}}\left(\int |W|^{\frac{2N}{N-2}}\right)^{\frac{2}{N}}-\int |f|^{\frac{2N}{N-2}}.
$$
Observing that the function
$$
\varphi:y \mapsto \left(\int |W|^{\frac{2N}{N-2}}\right)^{\frac{2}{N}} y^{\frac{N-2}{N}}-y
$$
is continuous and strictly positive on $\left(0,\int |W|^{\frac{2N}{N-2}}\right)$, we get \eqref{coerc}. The inequality \eqref{positive_E} is an immediate consequence of \eqref{coerc}.
\end{proof}
We divide the proof of Theorem \ref{T:crit_scatt2} into two propositions.
\begin{proposition}
 \label{P:compactness}
Assume that there exists $E_0\in \cR$ and a positive number $A<\int |W|^{\frac{2N}{N-2}}$ such that $\cS(E_0,A)=+\infty$. Then there exists a solution $u_c$ of \eqref{NLScrit} with maximal interval of existence $I_{\max}$, and functions $t\mapsto \lambda(t)\in (0,+\infty)$ and $t\mapsto x(t)\in \cR^N$, defined on $I_{\max}$ such that
\begin{equation}
\label{defK}
K=\Big\{ \frac{1}{\lambda(t)^{\frac{N-2}{2}}} ~ u_c\left(\frac{x-x(t)}{\lambda(t)},t\right),\; t\in I_{\max}\Big\}
\end{equation}
has compact closure in $\dot{H}^1(\cR^N)$ and satisfies
$\sup_{t\in I_{\max}} \int_{\cR^N} |u_c(t)|^{\frac{2N}{N-2}}\leq A$
and $E[u_c]>0$.

If $N=3,4$, then one can assume that $u_c$ is radial and $x(t)=0$ for all $t$.
\end{proposition}

\begin{proposition}
 \label{P:rigidity}
There exist no solution $u_c$ of \eqref{NLScrit} satisfying the conclusion of Proposition \ref{P:compactness}.
\end{proposition}

\begin{proof}[Sketch of proof of Proposition \ref{P:compactness}]

\noindent\emph{Step 1.} We first notice that by purely variational arguments and the small data theory, if $A<\int |W|^{\frac{2N}{N-2}}$ and $E_0>0$ is small, then $\cS(E_0,A)$ is finite.
Indeed, if $u$ is a solution of  \eqref{NLScrit} such that $E[u]\leq E_0<E[W]$ and $\int |u_0|^{\frac{2N}{N-2}}<\int |W|^{\frac{2N}{N-2}}$, then by Claim \ref{Cl:variational},
$\int |\nabla u_0|^2<\int |\nabla W|^2$, which, combined with the inequality $E[u]\leq E_0<E[W]$, implies, by Claim 2.6 of \cite{DM} that
$\int |\nabla u_0|^2\leq N E_0$, and the fact that $\cS(E_0,A)$ is finite follows from the small data theory.

\noindent\emph{Step 2.} We next construct the critical element $u_c$. Let $A<\int |W|^{\frac{2N}{N-2}}$ and assume that $\cS(E_0,A)=+\infty$ for some $E_0\in \cR$. Consider
$$
E_c=E_c(A)=\inf \{E_0\in \cR\text{ s.t. } \cS(E_0,A)=+\infty\}.
$$
Note that by the preceding step, $E_c$ is well defined and positive. We will prove the existence of $u_c$ as a consequence of the following lemma, analogous to Proposition 3.1 of \cite{KiVi}:
\begin{lemma}
\label{L:compactness}
Let $I_n=(T_n^-,T_n^+)$ be a sequence of intervals containing $0$. Let $\{u_n\}_n$ be a sequence of solutions of \eqref{NLScrit} on $I_n$, with initial data $u_{0,n}\in \dot{H}^1(\cR^N)$ at $t=0$, such that $u_n$ is radial if $N=3,4$ and
\begin{gather}
\label{S_seq}
 \lim_{n\to +\infty}\|u_n\|_{S(T_n^-,0)}=\lim_{n\to\infty}\|u_n\|_{S(0,T_n^+)}=+\infty\\
\label{energy_seq}
\lim_{n\to\infty} E[u_n]=E_c\\
\label{norm_seq}
\sup_{t\in I_n} \int|u_n(x,t)|^{\frac{2N}{N-2}}\,dx\leq A.
\end{gather}
Then there exists a subsequence of $\{u_{0,n}\}_n$ (still denoted by $\{u_{0,n}\}_n$) and sequences $\{x_n\}_n$, $\{\lambda_n\}_n$ such that
$\left\{\frac{1}{\lambda_n^{1/2}} ~ u_{0,n}\left(\frac{\cdot-x_n}{\lambda_{n}}\right)\right\}_n$ converges in $\dot{H}^1$.
\end{lemma}
(Of course, if $N=3,4$ in the lemma, we can assume $x_n=0$ for all $n$).

We omit the proof of Lemma \ref{L:compactness}, which is close to the one of \cite[section 3]{KM} and the proof of Proposition 3.1 of \cite{KiVi}. The main ingredients of the proof are the critical profile decomposition of Keraani \cite{Keraani}, long-time perturbation arguments and the criticality of $E_c$. We note that by Claim \ref{C:variational}, any nonzero profile in a profile decomposition of $\{u_{0,n}\}_n$ has strictly positive energy, which is crucial in the argument.

Let us assume Lemma \ref{L:compactness} and conclude the proof of Proposition \ref{P:compactness}. By the definition of $E_c$, there exists a sequence of intervals $\{I_n\}_n$ and a sequence of solutions $\{u_n\}_n$ of \eqref{NLScrit} on $I_n$ such that
\begin{gather}
 \label{CVenergy}
\lim_{n\to\infty} E[u_n]=E_c\\
\label{BupS}
\lim_{n\to\infty} \|u_n\|_{S(I_n)}=+\infty\\
\label{boundLp}
\forall n,\; \forall t\in I_n, \quad\int |u_n(x,t)|^{\frac{2N}{N-2}}\,dx\leq A.
\end{gather}
Time translating $u_n$ if necessary, we may assume by \eqref{BupS} that $I_n=(\theta_n^-,\theta_n^+)$ with $\theta_n^-<0<\theta_n^+$ and
\begin{equation}
 \label{BupSbis}
\lim_{n\to\infty} \|u_n\|_{S(\theta_n^-,0)}=\lim_{n\to\infty}\|u_n\|_{S(0,\theta_n^+)}=+\infty.
\end{equation}
By Lemma \ref{L:compactness} (with $T_n^-=\theta_n^-$, $T_n^+=\theta_n^+$) extracting a subsequence in $n$, rescaling and space-translating $u_n$, we can assume that there exists $u_{0,c}\in \dot{H}(\cR^N)$ such that
\begin{equation}
 \label{CVun}
\lim_{n\to\infty} \|u_{0,n}-u_{0,c}\|_{\dot{H}^1}=0.
\end{equation}
Let $(T_-,T_+)$ be the maximal interval of existence of $u_c$. Then by the continuity of the flow of \eqref{NLScrit},
\begin{equation}
\label{time_bound}
\liminf \theta_n^+\geq T_+,\quad \limsup \theta_n^-\leq T_-.
\end{equation}
Furthermore,
\begin{equation}
\label{uc_non_scatt}
 \|u_c\|_{S(0,T_+)}=\|u_c\|_{S(T_-,0)}=+\infty.
\end{equation}
Indeed, assume for example that $\|u_c\|_{S(0,T_+)}$ is finite. Then $T_+=+\infty$ and for large $n$, $u_n$ is globally defined forward in time and satisfies $\|u_n\|_{S(0,+\infty)} \leq 1+\|u_c\|_{S(0,+\infty)}$, a contradiction with \eqref{BupSbis}. By \eqref{CVenergy}, we get
\begin{equation}
\label{good_energy}
E[u_c]=E_c\geq E[W].
\end{equation}
Furthermore, by \eqref{boundLp} and \eqref{time_bound},
\begin{equation}
 \label{boundLp_uc}
\sup_{t\in (T_-,T_+)} \int |u_c(x,t)|^{\frac{2N}{N-2}}\,dx\leq A.
\end{equation}
Let $\{t_n\}_n$ be a sequence in $(T_-,T_+)$. By \eqref{uc_non_scatt},\eqref{good_energy} and \eqref{boundLp_uc},
the sequence of solutions $\{u_c(t_n+\cdot)\}_n$ satisfies the assumptions of Lemma \ref{L:compactness}
with $T_n^-=T_-$ and $T_n^+=T_+$, which shows that there exist sequences $\{\lambda_n\}_n$ and $\{x_n \}_n$
such that a subsequence of  $\left\{\frac{1}{\lambda_n^{1/2}}u_c\left(t_n,\frac{x-x_n}{\lambda_n}\right)\right\}_n$
converges in $\dot{H}^1$. By a standard lifting Lemma (e.g., see \cite[Appendix A]{DHR}),
one can deduce the existence of $\lambda(t)$ and $x(t)$ such that $K$ (defined by \eqref{defK}) has compact closure, which concludes the proof of Proposition \ref{P:compactness}.
\end{proof}
\begin{proof}[Proof of Proposition \ref{P:rigidity}]
We divide the proof into three parts, following again very closely \cite{KM} and, in Part 3, \cite{KiVi}. For simplicity, we will often omit the subscript $c$ and write $u=u_c$.

\noindent\emph{Part 1.} We show in this step that $u_c$ is global. Assume, for example, that $T_+(u_c)$ is finite. Let $\varphi\in C_0^{\infty}(\cR^N)$ such that $\varphi(x)=1$ if $|x|\leq 1$ and $\varphi(x)=0$ if $ |x|\geq 2$. Let
$$
M_R(t)=\int \varphi\left(\frac{x}{R}\right) |u(x,t)|^2\,dx.
$$
Then using that $u$ is bounded in $\dot{H}^1$ and Hardy's inequality,
$$
|M_R'(t)|=\left|2 \I \int \frac{1}{R}\nabla \varphi\left(\frac{x}{R}\right)\cdot\nabla u\,\overline{u}\,dx\right|\leq C\|\nabla u(t)\|^2\leq C,
$$
where $C$ is independent of $t$ and $R$. Thus, if $0\leq s<t<T_+(u)$, $R\geq 1$,
\begin{equation}
 \label{eq16}
\left|M_R(s)-M_R(t)\right|\leq C|t-s|.
\end{equation}
We next notice that there exists $t_n\to T_+(u)$ such that
\begin{equation}
 \label{eq17}
\lim_{n\to\infty} M_R(t_n)=0.
\end{equation}
Indeed, if $|x(t)|+\lambda(t)+\lambda(t)^{-1}$ is bounded as $t\to T_+(u)$, then by the compactness of $\overline{K}$, there exists a sequence $t_n\to T_+(u)$ such that $u(t_n)$ converges in $\dot{H}^1$, contradicting the fact that $T_+(u)$ is the maximal time of existence of $u$. Thus, there exists a sequence $t_n\to T_+(u)$ such that one of the following holds:
$$
\lim_{n\to\infty} |x(t_n)|=+\infty\text{ or }\lim_{n\to\infty}\lambda(t_n)=0\text{ or }\lim_{n\to\infty}\lambda(t_n)=+\infty.
$$
In each case, \eqref{eq17} follows easily.

Combining \eqref{eq16} and \eqref{eq17}, we see that for all $R\geq 1$ and for all $s\in [0, T_+(u))$,
$$
M_R(s)\leq C|T_+(u)-s|.
$$
Letting $R\to \infty$, we get by conservation of mass that $u_0\in L^2$ and that $M(u_0)\leq C|T_+(u)-s|$. Letting $s\to T_+(u)$, we get that $u_0=0$, contradicting the fact that the energy of $u$ is positive.

Note that to show $T_+(u_c)=+\infty$, we only used that
\begin{equation}
 \label{defK+}
K_+=\Big\{ \frac{1}{\lambda(t)^{\frac{N-2}{2}}}u_c\left(\frac{x-x(t)}{\lambda(t)},t\right),\; t\in [0,T_+(u_c))\Big\}
\end{equation}
has compact closure in $\dot{H}^1(\cR^N)$.

We next treat the global case. Let
$$
a=\inf_{t\in \cR} \int |u(x,t)|^{\frac{2N}{N-2}}.
$$
We note that by compactness of $\overline{K}$, $a>0$. We let $\eps_0=\eps_0(a,A)$ given by Claim \ref{C:variational}.
We distinguish between space dimensions $N=3, 4$ and $N\geq 5$.

\emph{Part 2.} Global radial case, $N=3, 4$. Here, $x(t)=0$ for all $t$.

We first assume
\begin{equation}
  \label{lambda_bnd_below}
\inf_{t\in [0,+\infty)}  \lambda(t)>0.
\end{equation}
Let $R$ be a large constant to be specified later, and
\begin{equation}
\label{def_zR}
z_R(t)=\int R^2\, \chi\left(\frac{x}{R}\right)|u(x,t)|^2\,dx,
\end{equation}
where $\chi$ is smooth, $\chi(x)=|x|^2$ if $|x|\leq 1$ and $\chi(x)=0$ if $|x|\geq 2$. Then
\begin{equation}
 \label{zR'}
|z_R'(t)|=\left|2 \I \int R\, \nabla \chi\left(\frac{x}{R}\right)\cdot\nabla u\,\overline{u}\,dx\right|.
\end{equation}
Since $u$ is bounded in $\dot{H}^1$, there exists $C>0$ such that
\begin{equation}
 \label{eq18}
\forall t \in [0,+\infty), \quad \left|z'_R(t)\right|\leq CR^2.
\end{equation}
By an explicit computation, using that $u$ is a solution of \eqref{NLScrit}, we get
\begin{equation}
 \label{eq19}
z''_R(t)\geq \underbrace{8\left(\int |\nabla u(t)|^2-\int |u|^{\frac{2N}{N-2}}\right)}_{(A)}-\underbrace{C\int_{|x|\geq R} |\nabla u|^2+\frac{|u|^2}{|x|^2}+|u|^{\frac{2N}{N-2}}}_{(B)}.
\end{equation}
By Claim \ref{C:variational}, the term $(A)$ in \eqref{eq19} is greater than $8\eps_0$. By the compactness of $\overline{K}$, one can chose $R$ large so that $|(B)|\leq \eps_0$. Combining, we get that if $R$ is large,
\begin{equation}
 \label{eq20}
\forall t\geq 0,\quad z''_R(t)\geq 7\eps_0.
\end{equation}
Integrating \eqref{eq20} between $0$ and $T>0$, we get
$$
C R^2 \geq z'_R(T)\geq 7\eps_0 T+z_R'(0),
$$
a contradiction if $T\to +\infty$, $R$ being fixed. This concludes the proof when \eqref{lambda_bnd_below} holds. Again we only used that $K_+$ defined by \eqref{defK+} has compact closure in $\dot{H}^1$.

We next assume that \eqref{lambda_bnd_below} does not hold. Using the compactness of $\overline{K}$, one can construct another solution $\tilde{u}_c$ of \eqref{NLScrit} such that $0\in I_{\max}(\tilde{u}_c)$,
\begin{equation*}
E[u_c]=E[\tilde{u}_c],\quad
\sup_{t\in [0,T_+(\tilde{u}_c))} \int |\tilde{u}_c|^{\frac{2N}{N-2}}\leq A
\end{equation*}
and there exists $\tilde{\lambda}(t)$ such that
$$
\tilde{K}_+=\left\{\frac{1}{\tilde{\lambda}^{\frac{N-2}{2}}(t)}\tilde{u}_c \left(\frac{x}{\tilde{\lambda}(t)},t\right),\; t\in [0,T_+(\tilde{u}_c))\right\}
$$
has compact closure in $\dot{H}^1$ and
$$
\inf_{t\in [0,T_+(\tilde{u}_c))} \tilde{\lambda}(t)>0.
$$
We refer to the proof of Theorem 5.1 in \cite{KM} for the construction of $\tilde{u}_c$ and $\tilde{\lambda}$.

By Part 1 of the proof, $T_+(\tilde{u}_c)=+\infty$. We are thus reduced to the case where \eqref{lambda_bnd_below} holds, concluding this part.

\emph{Part 3.} Global case, $N\geq 5$, without radial assumption.

By Part 1 of the proof, we can assume again that $u$ is globally defined. By Section 4 of \cite{KiVi}, we can assume that one of the following holds:
\begin{gather}
 \label{soliton}
\forall t\in \cR, \quad \lambda(t)=1\\
\label{cascade}
\sup_{t\in \cR} \lambda(t)<\infty\text{ and }\lim_{t\to +\infty} \lambda(t)\to 0
\end{gather}
By Theorem 6.1 of \cite{KiVi}, in both cases
\begin{equation}
 \label{u0L2}
u_0\in L^2(\cR^N).
\end{equation}
According to Theorem 7.1 of \cite{KiVi}, \eqref{cascade} does not hold. In this case we do not need the assumption that $\sup_t \int |u(t)|^{\frac{2N}{N-2}}< \int |W|^{\frac{2N}{N-2}}$.

We next assume that \eqref{soliton} holds. By Lemma 8.2 of \cite{KiVi}, $K$ defined by  \eqref{defK} (with $\lambda(t)=1$) has compact closure in $H^1(\cR^N)$. Applying the Galilean transform
$$
u(x,t) \mapsto e^{ix\cdot \xi_0}e^{-it|\xi_0|^2}u(x-2\xi_0t,t),
$$
with a suitable choice of $\xi_0$, one can assume that the conserved momentum
$ \im \int \nabla u \overline{u}$ is zero. Following \cite{DHR}, one can deduce
\begin{equation}
 \label{eq23}
\lim_{t\to \infty} \frac{x(t)}{t}=0.
\end{equation}
Let $C_0>0$ be a large constant (depending only on $a$ and $A$). Let $T>0$ and
$$
R=C_0+\max_{0\leq t\leq T} |x(t)|.
$$
Consider $z_R(t)$ defined by \eqref{def_zR}. Using that $u(t)$ is bounded in $H^1(\cR^N)$ we obtain that there is a constant $C>0$ independent of $R$ such that
\begin{equation}
 \label{eq24}
\forall t\in \cR,\quad |z_R'(t)|\leq CR.
\end{equation}
Furthermore, as before
\begin{equation}
 \label{bnd_zR''}
z''_R(t)\geq \underbrace{8\left(\int |\nabla u(t)|^2-\int |u|^{\frac{2N}{N-2}}\right)}_{(A)}-\underbrace{C\int_{|x|\geq R} |\nabla u|^2+\frac{|u|^2}{|x|^2}+|u|^{\frac{2N}{N-2}}}_{(B)}.
\end{equation}
Using Claim \ref{C:variational}, the compactness of $\overline{K}$ in $\dot{H}^1$ and the choice of $R$, we get, for $C_0$ large (independently of $0$ and $T$),
$$
\forall t\in [0,T], \quad z_R''(t)\geq 7\eps_0.
$$
Integrating between $0$ and $T$, we deduce
$$
C R \geq 7\eps_0 T,
$$
that is
$$
C \left(C_0+\max_{0\leq t\leq T}|x(t)|\right)\geq 7\eps_0|T|,
$$
which contradicts \eqref{eq23}, which concludes this sketch of proof.

We note that we could have (as in Part 2) reduced to a critical solution $u$ satisfying
$$
\inf_{t\geq 0} \lambda(t)>0,
$$
however, such a solution does not necessarily satisfy $\sup_{t}\lambda(t)<\infty$, a condition that is needed in \cite{KiVi} to prove that $u_0\in L^2$.
\end{proof}

\subsection{Energy-subcritical case}
\label{SS:subcrit}
Now we consider the NLS equation \eqref{E:NLS} when $0<s_c<1$ and
obtain scattering for bounded solutions in Theorem \ref{T:BB}:
\begin{theorem}
\label{T:scatt_sub1}
Let $u$ be a solution of \eqref{E:NLS}, where $p$ satisfies \eqref{E:p-subcrit},
and assume that $T_+(u)=+\infty$ and
$$
\limsup_{t\to +\infty} \left(\int |u(t)|^{p+1}\right)^{s_c}M[u]^{1-s_c}< \left(\int |Q|^{p+1}\right)^{s_c} M[Q]^{1-s_c}.
$$
Then $u$ scatters forward in time in $H^1$.
\end{theorem}
As in the previous subsection, we first state a slightly stronger result.
Define
$$
\|u\|_{S(I)}=\|u\|_{L^{\alpha}\left(I,L^{p+1}(\cR^N)\right)},
$$
where $\alpha=\frac{2(p-1)(p+1)}{4-(N-2)(p-1)}$.
Observe\footnote{In \cite[section 2.2.1]{G}, $S(I)$ is defined as the intersection of all $L^q(I,L^r)$ spaces with $(q,r)$ $\dot{H}^{s_c}$-admissible. It is sufficient for this paper to use just one such admissible paper, as in \cite{C+chinese} for example.}
that $(\alpha, p+1)$ is $\dot{H}^{s_c}$-admissible,
i.e., $\frac2{\alpha}+\frac{N}{p+1} = \frac{N}2-s_c$.
We note that if $u$ is a solution of \eqref{E:NLS} which is bounded in $H^1$ on $[0,+\infty)$ and such that $\|u\|_{S(0,+\infty)}$ is finite, then $u$ scatters forward in time (see \cite{bookC}).

For $L\in \cR$,  $A> 0$, we let $\cS(L,A)$ be the supremum of all $\|u\|_{S(I)}$, where $I$ is a real interval, and $u$ a solution of \eqref{E:NLS} on $I\times \cR^N$ such that
$$
\sup_{t\in I} \left(\int |u(t)|^{p+1}\right)^{s_c}M[u]^{1-s_c}\leq A\quad \text{and}\quad  E[u]^{s_c}M[u]^{1-s_c}\leq L.
$$
\begin{theorem}
\label{T:scatt_sub2}
If \eqref{E:p-subcrit} holds and $A<\left(\int |Q|^{p+1}\right)^{s_c}M[Q]^{1-s_c}$, then for all $L$ in $\cR$,
$$
\cS(L,A)<\infty.
$$
\end{theorem}
The proof is very close to the one of Subsection \ref{SS:crit}, but two things are simpler in the subcritical setting: all solutions that are bounded in $H^1$ are global, and there is no need for the scaling parameter $\lambda(t)$. The adaptation of the arguments of \cite{KM} in the critical case to a radial subcritical setting (cubic equation in dimension $3$), was done in \cite{CMP}. The radiality assumption was removed in \cite{DHR}. We refer to \cite{G} (and also to \cite{C+chinese}) for a general energy-subcritical and mass-supercritical NLS equation.

We start by proving the analog of Claim \ref{C:variational}, which is the only new ingredient of the proof. We will then state the analogs of Proposition \ref{P:compactness} and \ref{P:rigidity}.

\begin{claim}
 \label{C:variational_sub}
Let $a,A$ be such that
$$
0 < a < A < \left(\int |Q|^{p+1}\right)^{s_c}M[Q]^{1-s_c}.
$$
Then there exists $\eps_0=\eps_0(a,A)$ such that, for all $f\in H^1(\cR^N)$, if
$$
a \leq \left(\int |f|^{p+1}\right)^{s_c}M[f]^{1-s_c}\leq A,
$$
the two following properties hold:
\begin{align}
\label{B2}
 \int |\nabla f|^2-\frac{N(p-1)}{2(p+1)}\int |f|^{p+1}&\geq \eps_0M[f]^{1-\frac 1{s_c}}\\
\label{B3}
E[f]&\geq \frac{\eps_0}{2}M[f]^{1-\frac{1}{s_c}}.
\end{align}
\end{claim}
\begin{proof}
By Pohozhaev equality \eqref{E:Pohozaev}, $\int |\nabla Q|^2 = \frac{N(p-1)}{2(p+1)} \int |Q|^{p+1}$.
Recalling the Gagliardo-Nirenberg inequality \eqref{E:sharpGN}, we have
\begin{multline}
\label{E:bound_GN}
M[f]^{\frac{1}{s_c}-1}\int |\nabla f|^2-\frac{N(p-1)}{2(p+1)}M[f]^{\frac{1}{s_c}-1}\int |f|^{p+1}\\ \geq \frac{1}{c_Q}M[f]^{\frac{1}{s_c}-1-\kappa}\left(\int |f|^{p+1}\right)^{\frac{4}{N(p-1)}}-\frac{N(p-1)}{2(p+1)}M[f]^{\frac{1}{s_c}-1}\int |f|^{p+1}=\frac{y^{\frac{4}{N(p-1)}}}{c_Q}-\frac{N(p-1)}{2(p+1)}y,
\end{multline}
where $y=\int |f|^{p+1}M[f]^{\frac{1}{s_c}-1}$. The function $y\mapsto \frac{y^{\frac{4}{N(p-1)}}}{c_Q} - \frac{N(p-1)}{2(p+1)}y$ has only one zero $y_*$ on $(0,\infty)$, such that $y_*^{1-\frac{4}{N(p-1)}} =\frac{2(p+1)}{N(p-1)c_Q}$, and is positive between $0$ and $y^*$. Since the inequality \eqref{E:bound_GN} is an equality when $f=Q$, we get $$
y^*=M[Q]^{\frac{1}{s_c}-1}\int |Q|^{p+1},
$$
and \eqref{B2} follows. Noting that
$$
E[f]\geq \frac{1}{2}\left(\int |\nabla f|^2-\frac{N(p-1)}{2(p+1)}\int |f|^{p+1}\right)
$$
(indeed $\frac{N(p-1)}{4}\geq 1$), we get \eqref{B3}.
\end{proof}
The following propositions are the energy-subcritical analogs of Propositions \ref{P:compactness} and \ref{P:rigidity}.
\begin{proposition}
 \label{P:compactness_sub}
Assume that there exists $A<\left(\int |Q|^{p+1}\right)^{s_c}M[Q]^{1-s_c}$, and $ L\in \cR$ such that $\cS(L,A)=+\infty$. Then there exists a global solution $u_c$ of \eqref{E:NLS}, and a function $t\mapsto x(t)$ defined on $\cR$ such that
$$
K=\left\{ u_c(x-x(t),t),\; t\in \cR\right\}
$$
has a compact closure in $H^1$.
\end{proposition}
\begin{proposition}
 \label{P:rigidity_sub}
There exist no solution $u_c$ satisfying the conclusion of Proposition \ref{P:compactness_sub}.
\end{proposition}
The proof of Proposition \ref{P:compactness_sub} goes along the same lines as the proof of
Proposition \ref{P:compactness}: first, by purely variational arguments and the small data theory,
one notice that $S(L,A)$ is finite if $A<\left(\int |Q|^{p+1}\right)^{s_c}M[Q]^{1-s_c}$ and $L$ is small.
Then, using a suitable profile decomposition (see \cite{G} or \cite{Keraani}),
one shows the analog of Lemma \ref{L:compactness} to prove the existence of $u_c$ and the compactness of its trajectory up to the translation parameter $x(t)$. We note that the fact that $u_c$ is bounded in $H^1$ implies (since nonlinearity is energy-subcritical) that it is a global solution.

The proof of Proposition \ref{P:rigidity_sub} is very close to Part 3 of the proof of Proposition \ref{P:compactness} in the case where \eqref{soliton} holds. Let us just mention the analog of \eqref{bnd_zR''}:
\begin{equation*}
z''_R(t)\geq 8\left(\int |\nabla u|^2-\frac{N(p-1)}{2(p+1)} \int |u|^{p+1}\right)-C\int_{|x|\geq R} |\nabla u|^2+\frac{1}{|x|^2}|u|^2+|u|^{\frac{2N}{N-2}},
\end{equation*}
which yields a contradiction in the same way as in the above proof, replacing Claim \ref{C:variational} by Claim \ref{C:variational_sub}.

\section{Blow up criteria} \label{S:blowupcriteria}

In this section we obtain two criteria for blow up in finite time: the first one is a generalization of Lushnikov's criteria \cite{Lu95} and the second one is the modification of the first approach where the generalized uncertainty principle is replaced by an interpolation inequality \eqref{E:interp1}. Note that both criteria are applicable in the case of the energy-supercritical NLS equations with positive energy. For a specific case of the focusing 3d cubic NLS equation see \cite[Sections 3.1 and 3.2]{HPR}.

\subsection{Proof of Theorem \ref{T:1}.}
\label{S:mechanics}
We first obtain a version of an uncertainty principle. By integration by parts
$$
\|u\|_{L^2}^2 = \frac1N \int (\nabla \cdot x)  |u|^2 \, dx
= - \frac2{N} \R \int (x \cdot \nabla u)  \bar u \, dx.
$$
Since $|z|^2 = |\R z|^2+|\I z|^2$, we have
$$
\frac{N^2}4 \|u\|_{L^2}^4 + \left| \I \int (x \cdot \nabla u) \bar u \, dx \right|^2
= \left| \int (x \cdot \nabla u) \bar u \, dx \right|^2
\leq \|xu\|_{L^2}^2 \|\nabla u\|_{L^2}^2,
$$
where the last one is by Cauchy-Schwarz. Recalling the variance and its first derivative from \eqref{E:variance}, we obtain the uncertainty principle
\begin{equation}
\label{E:1}
\frac{N^2}4 \, \|u\|_{L^2}^4 +  \left| \frac{V_t(t)}{4}  \right|^2 \leq V(t) \, \|\nabla u\|_{L^2}^2.
\end{equation}
Recalling the second (time) derivative of the variance
\begin{equation}
 \label{E:2}
V_{tt}(t) = 4N(p-1) E[u] - 4 (p-1)s_c \|\grad
u(t)\|^2_{L^2(\cR^N)},
\end{equation}
we substitute the bound on $\|\nabla u\|^2_{L^2}$ from \eqref{E:1} into \eqref{E:2} to obtain
\begin{equation}
 \label{E:5}
V_{tt}(t) \leq 4N(p-1) E[u] - N^2 (p-1) s_c \frac{(M[u])^2}{V(t)} - \frac{(p-1)s_c}4 \frac{|V_t(t)|^2}{V(t)}.
\end{equation}
We rewrite the equation \eqref{E:5} to remove the last term
with $V_t^2$ by making the substitution
\begin{equation}
\label{E:VtoB}
V = B^{\frac1{\alpha+1}}, \qquad \alpha = \frac{(p-1)s_c}4=\frac{N(p-1)+4}{8},
\end{equation}
and thus,
$$
V_t  = \frac1{\alpha+1} \, B^{-\frac{\alpha}{\alpha+1}} B_t \qquad \text{and} \qquad
V_{tt}  = -\frac{\alpha}{(\alpha+1)^2}\,
B^{-\frac{2\alpha+1}{\alpha+1}} B_t^2 + \frac1{\alpha+1}\,
B^{-\frac{\alpha}{\alpha+1}} B_{tt},
$$
which gives
$$
B_{tt} \leq 4(\alpha+1) N (p-1) E[u] B^{\frac\alpha{\alpha+1}} -(\alpha+1)N^2 (p-1)s_c 
{(M[u])^2}{B^{\frac{\alpha-1}{\alpha+1}}},
$$
or equivalently,
\begin{equation}
\label{E:6}
B_{tt} \leq \frac{N (p-1)(N(p-1)+4)}2 \left( E[u] \,
B^{\frac{N(p-1)-4}{N(p-1)+4}} - \frac{N s_c}{4} \,
{(M[u])^2} \, {B^{\frac{N(p-1)-12}{N(p-1)+4}}} \right),\quad t\in [0,T_+).
\end{equation}
We rescale $B$ as follows: let
$B(t)=B_{max}\, b(at)$, where
\begin{equation}
\label{eq*}
B_{max}=\left(\frac{N s_c (M[u])^2}{4 E[u]}\right)^{\frac{N(p-1)+4}{8}},
\quad a = \frac{8\sqrt 2}{\sqrt{N \,s_c}} \, \frac{E[u]}{M[u]}.
\end{equation}
Then letting $s=at$, we get
\begin{equation}
\label{E:b}
 \omega\,b_{ss}\leq b^{\gamma}-b^{\delta},\quad s\in [0,T^+/a),
\end{equation}
where
$$
\gamma=\frac{N(p-1)-4}{N(p-1)+4},\quad \delta=\frac{N(p-1)-12}{N(p-1)+4} \equiv 2\gamma-1,\quad \omega=\frac{64}{N(p-1)(N(p-1)+4)}.
$$
Note that, since $p>1+\frac{4}{N}$,
\begin{equation}
 \label{E:gamma_delta}
0<\gamma<1,\quad -1<\delta<\gamma.
\end{equation}
To analyze equation \eqref{E:b}, a mechanical analogy of
a particle moving in a field with a potential barrier is used as it was adapted
in \cite{HPR} from work of Lushnikov \cite{Lu95}.  We rewrite \eqref{E:b} as
\begin{equation}
\label{E:b'}
\omega\,b_{ss}+\frac{\partial U}{\partial b}\leq 0,\quad t\in [0,T^+/a),
\end{equation}
where $U(b)=\frac{b^{\delta+1}}{\delta+1}-\frac{b^{\gamma+1}}{\gamma+1}$.
The analogy from mechanics is as follows: Let $b=b(t)$ be a coordinate of a particle (of mass 1) with a motion under 2 forces: $b_{tt} = F_1 +F_2$, where $F_1 = -\frac{1}{\omega}\frac{\partial U}{\partial B}$, and $F_2 = -g^2(t)$ is some unknown external force which pulls the particle towards zero.
The collapse occurs when this particle reaches the origin in a finite period of time, i.e., when
$b(t)=0 \quad \text{for some} \quad 0<t<\infty$. If the particle reaches the origin without the force
$-g^2(t)$, then it should also reach the origin in the situation when this force is applied. We are thus lead to consider equation:
\begin{equation}
\label{E:conservative}
\omega\,b_{ss}+\frac{\partial U}{\partial b}=0.
\end{equation}
Define the energy of the particle
\begin{equation}
\label{E:E(t)}
\mathcal{E}(s) = \frac{\omega}{2}\,b_s^2(s) + U(b(s)),
\end{equation}
which is conserved for solutions of \eqref{E:conservative}.
Note that $\frac{\partial U}{\partial b}=b^{\delta}-b^{\gamma}$. Thus, in terms of dependence of $U$ on the particle's coordinate $b$, it is a bell-shaped function near $1$ (for positive $b$) with the local maximum $U_{max}=\frac{1}{\delta+1}-\frac{1}{\gamma+1}$ attained at $b=1$. Using conservation of the energy for \eqref{E:E(t)}, we obtain immediately two blow-up criteria for solutions of \eqref{E:conservative}:
\begin{enumerate}
\item
If $\mathcal{E}(0) < U_{max}$ and $b(0) < 1$ (to the left of the
bump), then (it does not matter what $b_s(0)$ is, since there in not
enough energy to escape this region) the particle fall onto the
origin, and collapse occurs.
\item
If $\mathcal{E}(0)>U_{max}$, then the particle can overcome the energy barrier. Indeed, by energy conservation the sign of $b_t$ does not change, and the condition $b_t(0)<0$ is sufficient to produce collapse.
\end{enumerate}
Proposition \ref{P:collapse} shows that these two sufficient conditions for blow-up in finite time remains valid in case of the equation \eqref{E:b} (as well as a third condition corresponding to the limit case $\mathcal{E}=U_{max}$).

\begin{proposition}
\label{P:collapse}
Let $b$ be a nonnegative solution of \eqref{E:b'} such that one of the following holds:
\begin{itemize}
\item[({\bf A})]
$\mathcal{E}(0)<U_{max}\text{ and }b(0)<1$,
\item[({\bf B})]
$\mathcal{E}(0) >U_{max}$ and $b_s(0) < 0$.
\item[({\bf C})]
$\mathcal{E}(0) =U_{max}$, $b_s(0) < 0$ and $b(0) < 1$.
\end{itemize}
Then $T_+<\infty$.
\end{proposition}
\begin{proof}
Multiplying equation \eqref{E:b'} by $b_s$, we get
\begin{equation}
 \label{E:energy}
b_s(s)>0\Longrightarrow \mathcal{E}_s(s)<0\qquad b_s(s)<0\Longrightarrow \mathcal{E}_s(s)>0.
\end{equation}
We argue by contradiction, assuming $T_+=T_+(u)=+\infty$.

We first assume {\bf (A)}. Let us prove by contradiction:
\begin{equation}
\label{E:ineg_b}
 \exists s>0,\quad b_s(s)<0.
\end{equation}
If not, $b_s(s)\geq 0$ for all $s$, and \eqref{E:energy} implies that the energy  decay. By {\bf (A)}, $\mathcal{E}(s)\leq \mathcal{E}(0)< U_{max}$ for all $s$. Thus, $|b(s)-1|\geq \eps_0$ (where $\eps_0>0$ depends on $\mathcal{E}(0)$) for all $t$. Since by {\bf (A)} $b(0)<1$, we obtain by continuity of $b$ that $b(s)\leq 1-\eps_0$ for all $s$. By equation \eqref{E:b}, we deduce $b_{ss}(s)\leq -\eps_1$ for all $s$, where $\eps_1>0$ depends on $\eps_0$. Thus, $b$ is strictly concave, a contradiction with the fact that $b$ is positive and $T_+=+\infty$.

We have proved that there exists $s\geq 0$ such that $b_s(s)\leq 0$. Letting
$$
t_1=\inf \big\{s\geq 0 : b_s(s)\leq 0\big\},
$$
we get by \eqref{E:energy} that the energy is nonincreasing on $[0,t_1]$. Thus, $\mathcal{E}(s)<\mathcal{E}(0)\leq U_{max}$ on $[0,t_1]$, which proves that $b(s)\neq 1$ on $[0,t_1]$. Since $b(0)<1$, we deduce by the intermediate value theorem that $b(t_1)<1$ and by \eqref{E:b} that $b_{ss}(t_1)<0$. Since $b_s(t_1)\leq 0$, an elementary bootstrap argument, together with equation \eqref{E:b} shows that $b(s)\leq 1-\eps_0$, $b_s(s)<0$ and  $b_{ss}(s)\leq -\eps_1$ for $s> t_1$, for some positive constants $\eps_0,\eps_1$. This is again a contradiction with the positivity of $b$.

We next assume {\bf (B)}. Let $t_1$ be an interval such that $b_s(s)<0$ on $[0,t_1]$. By \eqref{E:energy}, $\mathcal{E}$ is nondecreasing on $[0,t_1]$, and thus, $\mathcal{E}(s)\geq \mathcal{E}(0)>U_{max}$ for all $s$ on $[0,t_1]$. As a consequence, $\frac{1}{2}b_s(s)^2\geq \mathcal{E}(0)-U_{max}>0$ for all $s$ in $[0,t_1]$, which shows that $b_s(s)\leq -\sqrt{\mathcal{E}(0)-U_{max}}$ on $[0,t_1]$. Finally, an elementary bootstrap argument shows that the inequality $b_s(s)\leq -\sqrt{\mathcal{E}(0)-U_{max}}$ is valid for all $s\geq 0$, a contradiction with the positivity of $b$.

Finally, we assume {\bf (C)}. By bootstrap again, $b_s(s)<0$, $b(s)<1$ and $b_{ss}(s)<0$ for all positive $s$, proving again that $b$ is a strictly concave function, a contradiction.
\end{proof}
We now formulate the conditions {\bf (A)}, {\bf (B)} and {\bf (C)} in a concise manner, in the spirit of \cite{HPR}.

Define $v$ by $b= v^{\alpha+1}$, where as before $\alpha = \frac{s_c(p-1)}4$.
Note that $2\alpha+1=\frac{N(p-1)}{4}$, $\alpha+1=\frac{N(p-1)+4}{8}$, $(\alpha+1)(\delta+1)=2\alpha$,
$(\alpha+1)(\gamma+1)=2\alpha+1$, $\omega=\frac{2}{(2\alpha+1)(\alpha+1)}$. Thus,
$$
\mathcal{E}=\frac{\alpha+1}{2\alpha+1}{v'}^2v^{2\alpha}+\frac{\alpha+1}{2\alpha}v^{2\alpha}
-\frac{\alpha+1}{2\alpha+1}v^{2\alpha+1}
\qquad \mbox{and} \qquad
U_{max}=\frac{1}{2\alpha} \frac{\alpha+1}{2\alpha+1}.
$$
Let $k=\frac{(p-1)s_c}2=2\alpha$ and introduce the function
\begin{equation}
\label{E:functionf}
f(x) = \sqrt{\frac{1}{k \, x^{k}} + x - \Big(1+\frac1{k}\Big)}.
\end{equation}
Then
$$
\mathcal{E} < U_{max} \iff |v_s| < f(v).
$$

The condition {\bf (A)} equates to
$$
v(0) < 1 \quad \text{and} \quad -f(v(0)) < v_s(0) < f(v(0)),
$$
the condition {\bf (B)} holds if and only if
$$
v_s (0) < - f(v(0))
$$
and the condition {\bf (C)} means:
$$
v(0)<1 \quad \text{and} \quad v_s (0) = - f(v(0)).
$$

Merging the above three conditions together, we obtain
\begin{equation}
\label{E:merging}
v_s(0) <
\begin{cases}
+f(v(0)) & \text{if } v(0)< 1\\
-f(v(0)) & \text{if } v(0)\geq 1.
\end{cases}
\end{equation}
Finally, recalling that $V(t) = \left(B_{max}\,b(at)\right)^\frac1{1+\alpha}$, and thus, by \eqref{eq*}  $\ds V(t) = (B_{max})^{\frac8{N(p-1)+4}} \, v\left(\frac{8 \sqrt 2 \,E}{\sqrt{N s_c} \,M} \, t \right)$, we obtain
\begin{equation}
\label{E:finalThm1}
\frac{V_t(0)}{M} < \sqrt{8 N s_c} \, g\left(\frac{4}{N s_c}\, \frac{V(0) E}{M^2} \right),
\end{equation}
which is the desired statement of Theorem \ref{T:1} .

\begin{remark}
Observe that the function $f(x)$ can be written as
$$
f(x) = \sqrt{\frac1{k} \left(\frac1{x^k} - 1 \right) + x -1 }.
$$
The limiting cases are
\begin{equation}
\label{E:functionflimits}
\lim_{k \to 0} f(x) = \sqrt{x-1-\ln x} \quad \text{and} \quad \lim_{k \to \infty} f(x) = \left\{
\begin{array}{ll} \sqrt{x-1} & \text{if}~~ x>1\\0 & \text{if}~~ x=1 \\ \infty & \text{if}~~ 0<x<1. \end{array} \right.
\end{equation}
The graph of $f(x)$ for various values of parameter $k$ is given in Figure \ref{F:g-graph}.
\end{remark}

\subsection{Sharp constant for an interpolation inequality}
\label{S:interpolation}

Before we prove Theorem \ref{T:2}, we obtain an interpolation inequality:
\begin{proposition}
\label{P:interp} Assume $p>1$ and $N\geq 1$. The inequality
\begin{equation}
 \label{E:interp1}
\|u\|_{L^2} \leq C_{p,N} \, \left( \|xu\|_{L^2}^\frac{N(p-1)}2 \, \|u\|_{L^{p+1}}^{p+1} \right)^{1/\left(\frac{N(p-1)}2 + (p+1)\right)}
\end{equation}
holds with the sharp constant $C_{p,N}$ (depending on the nonlinearity $p$ and dimension $N$)
given by \eqref{E:C-anyN-anypSimple}.
Moreover, equality is achieved if and only if there exists $\beta\geq 0$, $\alpha>0$ such that $|u(x)|=\beta
\phi(\alpha x)$, where
$$
\phi(x) =
\begin{cases}
(1-|x|^2)^{\frac1{p-1}} & \mbox{if} \quad 0 \leq |x|\leq 1 \\
0 & \textrm{if} \quad |x|> 1 .
\end{cases}
$$
\end{proposition}
Proposition \ref{P:interp} was proved in \cite{HPR} in the case $p=3$ using variational arguments. We give a shorter, direct proof.
\begin{proof}
Let $R>0$ be a parameter to be specified later. Split the mass of $u$ as follows
\begin{equation}
\label{split}
 \int |u(x)|^2\,dx=\frac{1}{R^2}\int_{|x|\leq R} (R^2-|x|^2)|u(x)|^2\,dx+\frac{1}{R^2}\int_{|x|\leq R} |x|^2|u(x)|^2\,dx+\int_{|x|\geq R} |u(x)|^2\,dx.
\end{equation}
By H\"older's inequality we have
\begin{align}
\label{bound_A}
\frac{1}{R^2}\int_{|x|\leq R} (R^2-|x|^2)|u(x)|^2\,dx&\leq \frac{1}{R^2}\left( \int_{|x|\leq R} \left( R^2-|x|^2 \right)^{\frac{p+1}{p-1}}\,dx \right)^{\frac{p-1}{p+1}}\left( \int |u(x)|^{p+1}\,dx \right)^{\frac{2}{p+1}}\\
\notag
&=\frac{1}{R^2}\left( \int_{|y|\leq 1}\left( R^2-R^2|y|^2 \right)^{\frac{p+1}{p-1}}R^N\,dy \right)^{\frac{p-1}{p+1}}\|u\|^2_{L^{p+1}}\\
\notag
&=R^{\frac{N(p-1)}{p+1}}D_{p,N}\|u\|^2_{L^{p+1}},\quad \text{where }D_{p,N}=\left(\int_{|y|\leq 1}\left(1-|y|^2\right)^{\frac{p+1}{p-1}}\,dy\right)^{\frac{p-1}{p+1}}.
\end{align}
Note that
$$
D_{p,N}=\left( \sigma_N\int_0^1 (1-r^2)^{\frac{p+1}{p-1}}r^{N-1}\,dr \right)^{\frac{p-1}{p+1}},
$$
where $\sigma_N$ stands for the surface area of the ${N-1}$ dimensional sphere, i.e.,
$$
\sigma_N = \frac{2 \pi^{N/2}}{\Gamma(\frac{N}2)}, �\quad
\Gamma(s+1) = \int_0^\infty e^{-t} \,t^s \,dt.
$$
By the change of variable $s=r^2$,
\begin{equation*}
\int_0^1 (1-r^2)^{\frac{p+1}{p-1}}r^{N-1}\,dr=\frac{1}{2}\int_0^1 (1-s)^{\frac{p+1}{p-1}}s^{\frac{N}{2}-1}\,ds=\frac{\Gamma\left( \frac N2 \right)\Gamma \left( \frac{2p}{p-1}\right)}{2\,\Gamma\left( \frac{2p}{p-1}+\frac N2 \right)},
\end{equation*}
where we have used the property of the Beta distribution (e.g., see \cite[p.623]{CB} or \cite[p.396 \#586]{MathTable})
$$
\int_0^1 x^{a-1} (1-x)^{b-1} \, dx = \frac{\Gamma(a)\Gamma(b)}{\Gamma(a+b)}.
$$
Hence,
\begin{equation}
\label{E:DpN}
D_{p,N}=\left( \pi^{\frac{N}{2}} \frac{\Gamma\left( \frac{2p}{p-1} \right)}{\Gamma\left( \frac{2p}{p-1}  +\frac N2\right)} \right)^{\frac{p-1}{p+1}}.
\end{equation}
Furthermore,
\begin{equation}
\label{bound_B}
\frac{1}{R^2}\int_{|x|\leq R} |x|^2|u(x)|^2\,dx+\int_{|x|\geq R} |u(x)|^2\,dx\leq \frac{1}{R^2}\int |x|^2|u(x)|^2\,dx.
\end{equation}
Combining \eqref{split}, \eqref{bound_A},\eqref{E:DpN} and \eqref{bound_B}, we obtain
\begin{equation}
 \label{bound_mass}
\forall R>0,\quad
\|u\|^2_{L^2}\leq \left( \pi^{\frac{N}{2}} \frac{\Gamma\left( \frac{2p}{p-1} \right)}{\Gamma\left( \frac{2p}{p-1} +\frac N2\right)} \right)^{\frac{p-1}{p+1}}\|u\|^2_{L^{p+1}}R^{\frac{N(p-1)}{p+1}}+\frac{1}{R^2}\|xu\|_{L^2}^2.
\end{equation}
Noting that the minimum of the function $F(R)=A R^{\alpha}+BR^{-2}$ (with $A,B>0$ and $\alpha>0$), attained at $R=\left(\frac{2B}{\alpha A}\right)^{\frac{1}{\alpha+2}}$ on $(0,+\infty)$, is
$$
F_{min}=\frac{2+\alpha}{2\alpha} (\alpha A)^{\frac{2}{\alpha+2}}(2B)^{\frac{\alpha}{\alpha+2}},
$$
we deduce from \eqref{bound_mass} that
\begin{multline*}
\|u\|^2_{L^2}\leq \\
\left( \frac{N(p-1)+2(p+1)}{2N(p-1)} \right) \left( \frac{N(p-1)}{p+1}\left( \pi^{\frac{N}{2}} \frac{\Gamma\left( \frac{2p}{p-1} \right)}{\Gamma\left( \frac{2p}{p-1} +\frac{N}2 \right)} \right)^{\frac{p-1}{p+1}}\|u\|^2_{L^{p+1}}\right)^{\frac{2(p+1)}{N(p-1)+2(p+1)}}\left(2\, \|xu\|^2_{L^2}\right)^{\frac{N(p-1)}{N(p-1)+2(p+1)}}.
\end{multline*}
Using
$$
\Gamma\left( \frac{2p}{p-1} \right)=\frac{p+1}{p-1}\Gamma\left( \frac{p+1}{p-1} \right)  \quad \mbox{and} \quad \Gamma\left( \frac{2p}{p-1} +\frac{N}{2}\right)=\left(\frac{p+1}{p-1} +\frac{N}{2}\right)\,\Gamma\left( \frac{p+1}{p-1} +\frac{N}{2} \right),
$$
we get (after tedious but straightforward computations) the inequality \eqref{E:interp1} with
\begin{equation}
\label{E:C-anyN-anypSimple}
\left(C_{p,N} \right)^{\frac{N(p-1)}{2} + (p+1)}
= \frac{N}{2\pi} \left( \frac{p-1}{p+1} \right)
\left(\pi \Big(1+\frac{2(p+1)}{N(p-1)}\Big) \right)^{\frac{N(p-1)}4 + 1} \left(\frac{\Gamma\left(\frac{p+1}{p-1}\right)}{\Gamma\left(\frac{p+1}{p-1} + \frac{N}2\right)} \right)^{\frac{p-1}2}.
\end{equation}
Note that equality in \eqref{E:interp1} holds if and only if there exists $R>0$ such that \eqref{bound_mass} is an equality. This is equivalent to the fact that for some $R>0$, both \eqref{bound_A} and \eqref{bound_B} are equalities. The inequality \eqref{bound_A} is an equality if and only if $|u(x)|^{p+1}=c \left( R^2-|x|^2 \right)^{\frac{p+1}{p-1}}$ for some constant $c\geq 0$, and the inequality \eqref{bound_B} is an equality if and only if $u(x)=0$ for $|x|\geq R$. This completes the proof of Proposition \ref{P:interp}.
\end{proof}

For several specific cases we compute the sharp constant $C_{p,N}$ explicitly in Appendix \ref{A:sharp-constants}.

\subsection{Second approach and Proof of Theorem \ref{T:2}.}

Recall the energy
$$
E[u] = \frac12 \, \big\|\grad u(t)\big\|^2_{L^2(\cR^N)} -
\frac1{p+1} \, \big\|u(t)\big\|^{p+1}_{L^{p+1}(\cR^N)},
$$
then solving for the kinetic energy term and substituting into \eqref{V''2}, we obtain
$$
V_{tt}(t) = 16 E - \frac{8(p-1)s_c}{p+1}  \, \big\|u(t)\big\|^{p+1}_{L^{p+1}(\cR^N)}.
$$
Using the sharp interpolation inequality from \eqref{E:interp1}, we get
\begin{equation}
\label{E:V_tt2}
V_{tt}(t) \leq 16 E - \frac{8(p-1)s_c}{(p+1) \left(C_{p,N} \right)^{\frac{N(p-1)}{2} + (p+1)}} \frac{M^{\frac{N(p-1)}{4}+\frac{(p+1)}2}}{V^{\frac{N(p-1)}{4}}}
\end{equation}
with $C_{p,N}$ from \eqref{E:interp1}. We next apply the same mechanical approach as in Section \ref{S:mechanics}.

We introduce again rescaling: define $v(s)$ (with $s = at$) as
$$
V(t) = V_{max} v(a t), \quad a = \sqrt{\frac{32 E}{V_{max}}},
$$
where
$$
V_{max}=\left(\frac{s_c(p-1)}{2(p+1)}\right)^{\frac{4}{N(p-1)}} \frac{M^{1+\frac{2(p+1)}{N(p-1)}}}{(C_{p,N})^{2+\frac{4(p+1)}{N(p-1)}}\, E^{\frac{4}{N(p-1)}}}.
$$
Then the inequality in \eqref{E:V_tt2} becomes
$$
v_{ss} (s) \leq \frac12 \left(1 - v^{-\frac{N(p-1)}{4}}(s) \right).$$
When the inequality in the previous expression is replaced by an equality, we obtain that the following energy is conserved:
$$
\mathcal{E}(s) =\frac{k}{1+k} \left((v_s(s))^2 - v (s)  - \frac1{k v(s)^k}\right),
$$
where as before $k= \frac{(p-1)s_c}{2}=\frac{N(p-1)}{4}-1$. Note that the maximum of the function $x\mapsto -\frac{k}{1+k}\left(x+\frac{1}{kx^k}\right)$, attained at $x=1$, is $-1$.

Similar to {\bf (A)}, {\bf (B)} and {\bf (C)}, we identify the three sufficient conditions for blow up in finite time:
\begin{itemize}
\item[({\bf A$^*$})]
$\mathcal{E}(0) < -1$ and $v(0) < 1$.

\item[({\bf B$^*$})]
$\mathcal{E}(0) >-1$ and $\partial_s v(0) < 0$.

\item[({\bf C$^*$})]
$\mathcal{E}(0) =-1$, $\partial_s v(0) < 0$ and $v(0) < 1$.
\end{itemize}

Recalling the function $f$ from \eqref{E:functionf} and using the definition of $\mathcal{E}$, we obtain
\begin{itemize}
\item[]
$\mathcal{E} < -1 \iff |\partial_s v| < f(v)$,
\item[]
$\mathcal{E} \geq - 1 \iff |\partial_s v| \geq f(v)$.
\end{itemize}
Then condition {\bf (A$^*$)} holds if and only if
$$
v(0) < 1 \quad \text{and} \quad -f(v(0)) < \partial_s v(0) < f(v(0)),
$$
the condition {\bf (B$^*$)} holds if and only if
$$
\partial_s v(0) <- f(v(0)),
$$
and the condition {\bf (C$^*$)} holds if and only if
$$
\partial_s v(0) =- f(v(0)),\quad v(0)<1.
$$
Merging the three conditions together, we obtain
\begin{equation}
\label{E:merging2}
\partial_s v(0) <
\begin{cases}
 f(v(0)) & \text{if } v(0)<1\\
-f(v(0)) & \text{if } v(0)\geq 1.
\end{cases}
\end{equation}
Substituting back $V(t)$, we obtain
$$
\frac{V_t(0)}{4\sqrt2} \, \frac{\left(C_{p,N}\right)^{1+\frac{2(p+1)}{N(p-1)}}}{ E^{\frac{s_c}{N}}M^{\frac12+\frac{p+1}{N(p-1)}} } \left(\frac{2(p+1)}{s_c(p-1)} \right)^{\frac2{N(p-1)}} < g(\theta),
$$
where $g$ is defined in \eqref{E:g-def} and
$$
\theta = \left(\frac{2(p+1)}{s_c(p-1)} \left(C_{p,N}\right)^{\frac{N(p-1)}{2} + (p+1)}\right)^{\frac4{N(p-1)}}
\frac{E^{\frac4{N(p-1)}}}{M^{1+\frac{2(p+1)}{N(p-1)}}} V(0).
$$
Simplifying, we complete the proof of Theorem \ref{T:2}.

\subsection{Reformulation of blow up criteria for real-valued data} \label{S:real}
Both Theorems \ref{T:1} and \ref{T:2} are significantly simplified for the real-valued data,
since $V_t(0)=0$ and the criteria {\bf (B)} and {\bf (C)} are irrelevant.
Respectively,
\begin{align}
\label{E:real1}
\quad \mbox{Theorem \ref{T:1}} \quad \Leftrightarrow \quad &
V(0) < \frac{N s_c}4 \, \frac{M^2}{E}\\
\label{E:real2}
\quad \mbox{Theorem \ref{T:2}} \quad \Leftrightarrow \quad &
V(0) < K \, \frac{M^{\frac{2(p+1)}{N(p-1)}+1}}{E^{\frac4{N(p-1)}}}, \\
&\mbox{where} \quad K=\left( \frac{s_c(p-1)}{2(p+1)} \frac1{\left(C_{p,N}\right)^{\frac{N(p-1)}2 +(p+1)}} \right)^{\frac4{N(p-1)}}.
\end{align}
The second condition \eqref{E:real2} is an improvement over the first one \eqref{E:real1} when
\begin{equation}
\label{E:real-comparison}
M[u]^{1-s_c} E[u]^{s_c} > \left(\frac{N s_c}{4}\right)^{\frac{N}{2}}
\left(\frac{2(p+1)}{s_c(p-1)} \left(C_{p,N}\right)^{\frac{N(p-1)}2 +(p+1)} \right)^{\frac2{p-1}}.
\end{equation}

\section{Examples}
\label{S:Examples}

In this section we consider examples of initial data for which Theorems \ref{T:BB}, \ref{T:1} and \ref{T:2} describe global behavior in the energy-critical case, and then in the energy-supercritical case, we finish with a $1d$ example in the energy-subcritical case. We first review some properties of the stationary solution $W$, and then consider gaussian data.

\subsection{Useful properties of $W$, the energy-critical case.}
Recall that $W$ is a stationary solution of \eqref{E:NLS} in the energy-critical
case, given by
$$
W = \frac{1}{\left(1 + \frac{|x|^2}{N(N-2)}
\right)^{\frac{N-2}{2}}}, \quad N \geq 3,
$$
and hence, solves $\Delta W + W^{\frac{N+2}{N-2}} = 0$. We compute
$\|\nabla W\|^2$ and $E[W]$. By the work of Aubin \cite{A76} and Talenti \cite{T76}, $W$ is, up
to symmetries, the only solution of
$$
\left(\int_{\cR^N} W^{\frac{2N}{N-2}}\right)^{\frac{N-2}{2N}} = C_N
\left(\int_{\cR^N} |\nabla W|^2\right)^{\frac{1}{2}},
$$
where $C_N$ is the best constant for the Sobolev inequality
$$
\left(\int_{\cR^N} u^{\frac{2N}{N-2}}\right)^{\frac{N-2}{2N}} \leq
C_N \left(\int_{\cR^N} |\nabla u|^2\right)^{\frac{1}{2}}.
$$
The constant $C_N$ is known (see e.g. \cite{T76}), namely,
$$
C_N = \frac{1}{\sqrt{N(N-2)\pi}}
\left(\frac{\Gamma(N)}{\Gamma(N/2)}\right)^{1/N}.
$$
Furthermore, by a direct integration by parts and the equation
$\Delta W=-W^{\frac{N+2}{N-2}}$, we have
$$
\int |\nabla W|^2=\int W^{\frac{2N}{N-2}}.
$$
Thus,
$$
\left(\int |\nabla W|^2\right)^{\frac{N-2}{2N}} = C_N\left(\int
|\nabla W|^2\right)^{\frac 12}
$$
from which we deduce
$$
\int |\nabla W|^2=\frac1{C_N^N}.
$$
Hence,
\begin{proposition}
$$
\int |\nabla W|^2=(N(N-2)\pi)^{\frac{N}2} \times
\frac{\Gamma(N/2)}{\Gamma(N)}
$$
and
\begin{equation}
 \label{E:energyW}
E[W]=N^{\frac{N}2-1}((N-2)\pi)^{\frac{N}2}\times \frac{\Gamma(N/2)}{\Gamma(N)}.
\end{equation}
\end{proposition}

\subsection{Gaussian Initial Data} \label{S:energy-critical}
We consider the gaussian initial data:
\begin{equation}
 \label{E:gaussian}
u_g(x) = \beta \, e^{-\frac12 \,{\alpha} \, |x|^2},
\quad x \in \cR^N, \quad \alpha, \beta \in (0,\infty).
\end{equation}
The mass and initial variance are
\begin{equation}
M[u_g]  = \beta^2 \, \left(\frac{\pi}{\alpha}\right)^{\frac{N}{2}},
\quad V(0) = \frac{N \, \beta^2 \, \pi^{\frac{N}2}}{2 \, \alpha^{\frac{N}2+1}}.
\end{equation}
Since in the case $s=1$ the energy is
$$
E[u_g] = \frac{\pi^{N/2}}2 \left(\frac{\beta}{\alpha^{(N-2)/4}}\right)^2 \left(\frac{N}2 - \left(\frac{N-2}{N} \right)^{\frac{N+2}2} \left(\frac{\beta}{\alpha^{(N-2)/4}}\right)^{\frac4{N-2}}\right),
$$
we obtain that $E[u_g] > E[W]$ holds when
\begin{equation}
\label{E:s=1-E}
\kappa_s < \frac{\beta}{\alpha^{(N-2)/4}} < \kappa_b,
\end{equation}
where $\kappa_s$ and $\kappa_b$ are the smallest and the largest positive roots (there are only 2 positive roots) of the equation
\begin{equation}
\label{E:s=1-E-kappa}
\frac{N}4 \kappa^2 - \frac12 \left(\frac{N-2}{N}\right)^{\frac{N+2}2} \kappa^{\frac{2N}{N-2}} - N^{\frac{N-2}2} (N-2)^{\frac{N}2} \frac{\Gamma(N/2)}{\Gamma(N)} = 0,
\end{equation}
and several values are listed in Table \ref{T:s=1-E}.

\begin{table}[ht]
\begin{tabular} [c]{|l||l|l|l|l|l|l|}
\hline
$N$ & $N=3$ & $N=4$ & $N=5$  & $N=6$ & $N=7$ & $N=8$ \\
\hline
$\kappa_s$ & $1.0376$ & $1.8388$ & $3.4406$ & $6.8142$ & $14.1738$ & $30.7504$ \\
\hline
$\kappa_b$ & $2.0510$ & $3.5523$ & $6.5483$ & $12.8229$ & $26.4385$ & $56.9593$ \\
\hline
\end{tabular}
\smallskip

\caption{Values of $\kappa_s$ and $\kappa_b$ in \eqref{E:s=1-E}, the two positive roots in \eqref{E:s=1-E-kappa} for $N=3, 4, 5, 6, 7, 8$.}
  \label{T:s=1-E}
\end{table}

We now apply the blow up criteria from Theorems \ref{T:1} and \ref{T:2} to data \eqref{E:gaussian}:
by Theorem \ref{T:1}, or \eqref{E:real1}, we have
\begin{equation}
\label{E:s=1-T1}
\beta > \left(\frac{N}2 \left(\frac{N}{N-2}\right)^{\frac{N}2} \right)^{\frac{N-2}4} \, \alpha^{\frac{N-2}4} := \kappa_{T1} \, \alpha^{\frac{N-2}4},
\end{equation}
we list some of $\kappa_{T1}$ in Table \ref{T:s=1-T}.
By Theorem \ref{T:2}, or \eqref{E:real2},  we have
\begin{equation}
\label{E:s=1-T2}
\beta > \left(\frac{\frac{N}2 \left(\frac{N}{N-2}\right)^{\frac{N}2} }{\frac2{N} \left(\frac{N}{N-2}\right)^{\frac{N}2} \left(\frac{2\pi}{N}\right)^{\frac{N}{N-2}} \left(C_{p,N}\right)^{-\frac{4N}{N-2}} + \frac{N-2}{N}} \right)^{\frac{N-2}4} \, \alpha^{\frac{N-2}4} :=\kappa_{T2} \, \alpha^{\frac{N-2}4} ,
\end{equation}
where
$$
\left (C_{p,N}\right)^{\frac{4N}{N-2}} = 4^{\frac{N-1}{N-2}} \, \pi^{\frac{N}{N-2}} \, \left( \frac{\Gamma(N/2)}{\Gamma(N)} \right)^{\frac2{N-2}}
$$
is from \eqref{E:C-anyN-anypSimple} (with $p=\frac{N+2}{N-2}$) and $\kappa_{T2}$ is listed also in Table \ref{T:s=1-T}.

\begin{table}[h]
\begin{tabular} [c]{|l||l|l|l|l|l|l|}
\hline
$N$ & $N=3$ & $N=4$ & $N=5$  & $N=6$ & $N=7$ & $N=8$\\
\hline
$\kappa_{T1}$
& $1.6709$
& $2.8284$
& $5.1811$
& $10.1250$
& $20.8639$
& $44.9492$ \\
\hline
$\kappa_{T2}$
& $1.1511$
& $3.0237$
& $6.6738$
& $14.0643$
& $30.0765$
& $66.1524$ \\
\hline
\end{tabular}
\smallskip

\caption{Blow up threshold values from Theorems \ref{T:1} and \ref{T:2} for real gaussian for $N=3, 4, 5, 6, 7, 8$.}
  \label{T:s=1-T}
\end{table}

\smallskip

We now consider a few different cases.
\bigskip

\underline{Example: {$N=3, p=5$}}.
\medskip

 Note that $E[u_g] > E[W]$ ~~ if ~ (see Table \ref{T:s=1-E})
\begin{equation}
\label{E:s=1-ME}
1.0376 < \frac{\beta}{\alpha^{1/4}} < 2.0510.
\end{equation}

Thus, by Theorem \ref{T:KMD} (and also Theorem \ref{T:BB-realvalued} for the threshold) solutions scatter when $$\frac{\beta}{\alpha^{1/4}} \leq \kappa_s \approx 1.0376$$
and from the same results solutions blow up in finite time if $\frac{\beta}{\alpha^{1/4}} > \kappa_b \approx 2.0510$. From Table \ref{T:s=1-T}, we obtain that the blow up occurs when $$\beta > 1.6709 \alpha^{1/4}$$ (Theorem \ref{T:1}) or when $$\beta > 1.1511 \alpha^{1/4}$$ (Theorem \ref{T:2}). The last threshold provides the wider range, so Theorem \ref{T:2} gives a better result in this case. We plot all these thresholds in Figure \ref{F:s=1-N=3}.

\begin{figure}
\includegraphics[scale=1.1]{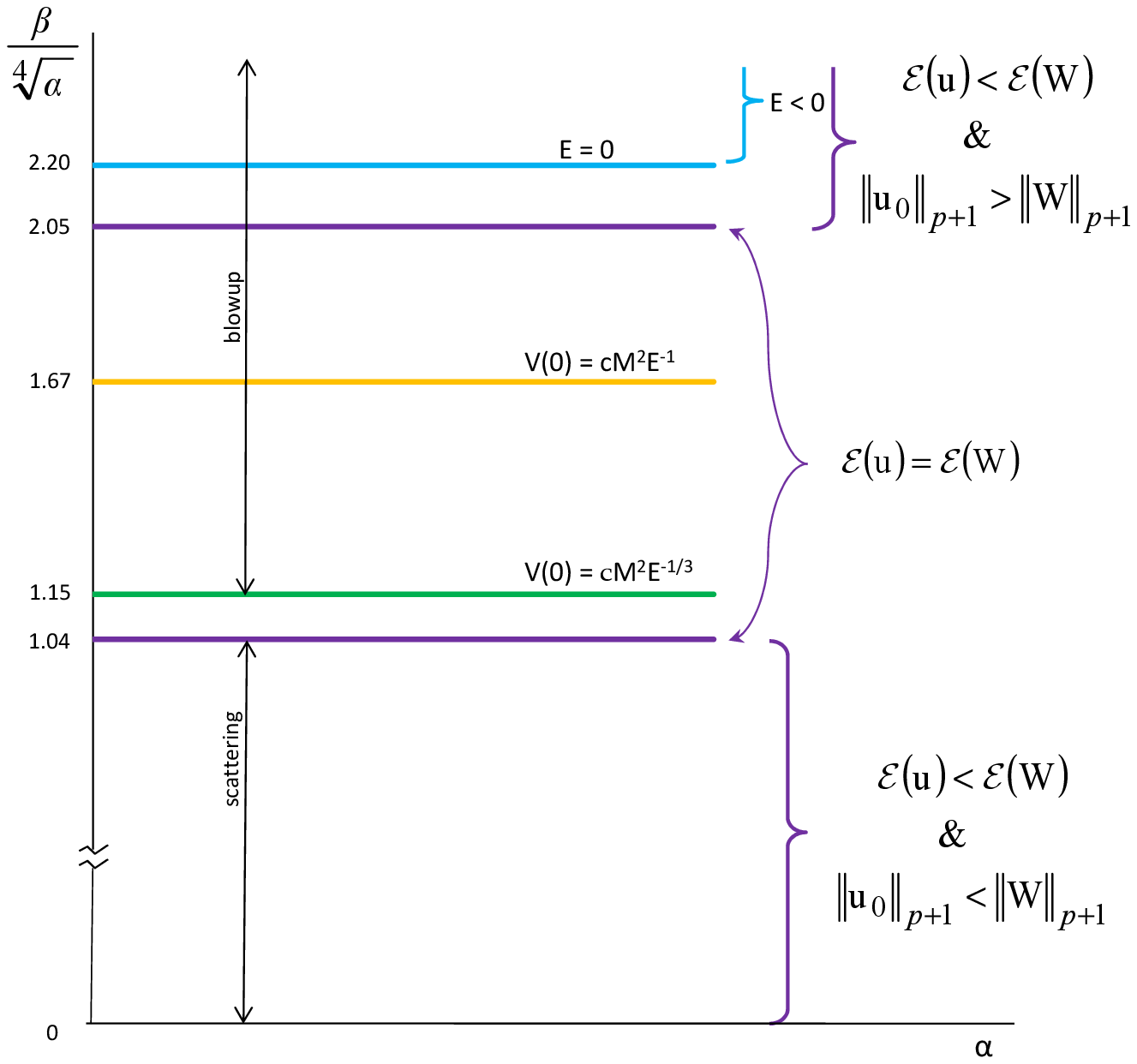}
\caption{Global behavior of solutions to the 3d quintic NLS ($s_c = 1$) with the gaussian initial data $u_0(x) = \beta \, e^{\alpha |x|^2 / 2}$. The top line ($E=0$) indicates the range above which (for $\beta> 2.20 \alpha^{1/4}$) solutions blow up in finite time by the negative energy criterion. The line denoted by $V(0)=c M^2 E^{-1}$ is the threshold for blow up from Theorem \ref{T:1}, see \eqref{E:s=1-T1} and Table \ref{T:s=1-T}; similarly, the line denoted by $V(0)=c M^{2} E^{-1/3}$ is the threshold for blow up from Theorem \ref{T:2}, see \eqref{E:s=1-T2} and Table \ref{T:s=1-T}. In this example, Theorem \ref{T:2} gives the best result. The range for {scattering} is given by Theorem \ref{T:KMD}. }
 \label{F:s=1-N=3}
\end{figure}

\newpage

\underline{Example: $N=4$, $p=4$}.
\medskip

\noindent
The condition $E[u_g]>E[W]$ holds when (see Table \ref{T:s=1-E})
\begin{equation}
\label{E:s=1-N=4-E}
1.8388 < \frac{\beta}{\sqrt\alpha} < 3.5523,
\end{equation}
and the blow up criteria (see Table \ref{T:s=1-T}) give $\beta > 2.8284 \sqrt\alpha$ (by Theorem \ref{T:1}) and $\beta > 3.0237\sqrt\alpha$ (by Theorem \ref{T:2}), thus, Theorem \ref{T:1} produces a better result in this case.
We plot the ranges for blow up and scattering in Figure \ref{F:s=1-N=4}.

\begin{figure}
\includegraphics[scale=1]{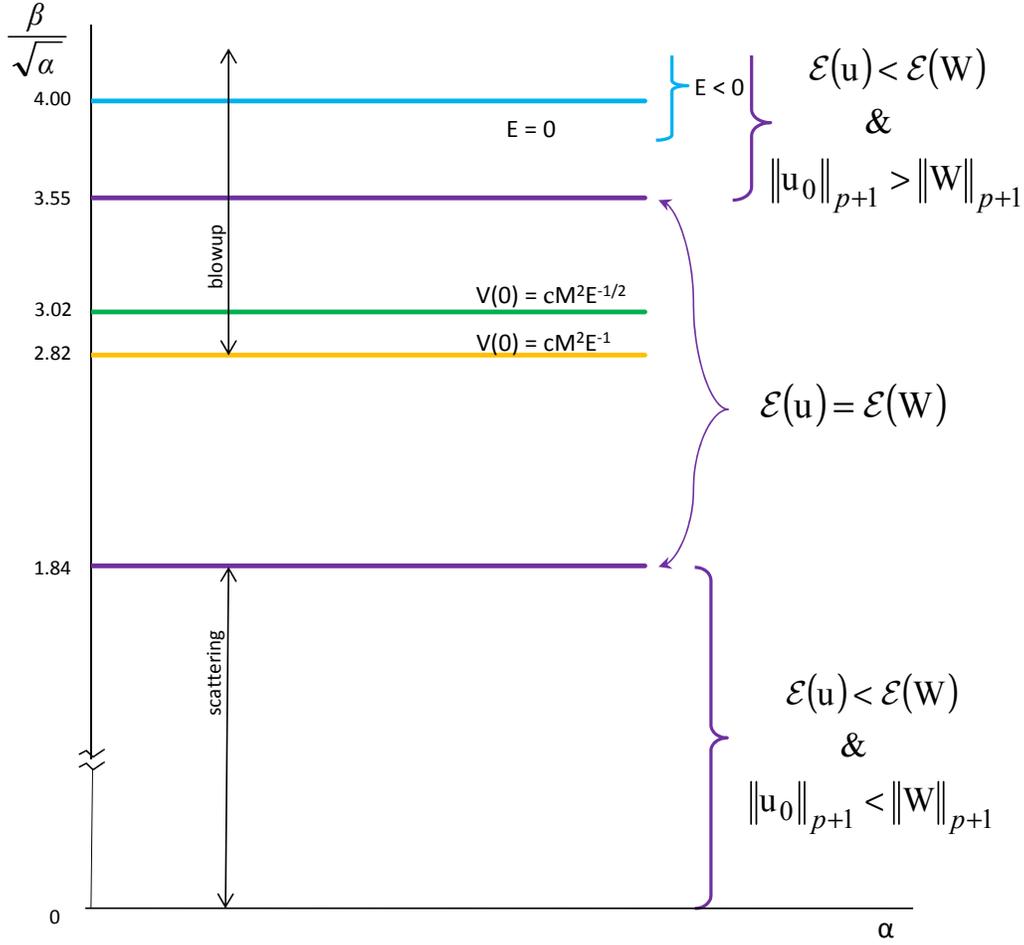}  %
\caption{Global behavior of solutions to the 4d cubic NLS ($s_c = 1$) with the gaussian initial data $u_0(x) = \beta \, e^{\alpha |x|^2 / 2}$. The top line ($E=0$) indicates the range above which (for $\beta> 4 \sqrt\alpha$) solutions blow up by the negative energy criterion. The line denoted by $V(0)= c M^2 E^{-1}$ is the threshold for blow up from Theorem \ref{T:1}, similarly, the line denoted by $V(0)=c M^{2} E^{-1/2}$ is the threshold for blow up from Theorem \ref{T:2}. In this example, Theorem \ref{T:1} gives a wider range, i.e., solutions blow up when $\beta \geq 2.82 \sqrt\alpha$. The range for {scattering} is given by Theorem \ref{T:KMD}. }
  \label{F:s=1-N=4}
\end{figure}

\medskip

\newpage

\subsubsection{The energy-supercritical case.}\label{S:energy-supercritical}
We now point out that both Theorems \ref{T:1} and \ref{T:2} work in the energy-supercritical case of NLS ($s>1$), and thus, it is possible to predict a blow up behavior for an open set of initial data in this case.
\medskip

\underline{Example $s_c>1$: $p=5$, $N=4$}.
\medskip

\noindent In this case $s_c = \frac32$.
The standard negative energy condition $E[u_g] < 0$ produces the range $\beta > \sqrt[4]{54} \, \alpha^{1/4} \approx 2.7108 \, \alpha^{1/4}$ for initial data to blow up in finite time.
By Theorem \ref{T:1}, or \eqref{E:real1}, we have $V(0) < 1.5 M^2 E^{-1}$, which gives the range
\begin{equation}
\label{E:s>1-T1}
\frac{\beta}{\alpha^{1/4}} > \frac{3^{3/4}}{2^{1/4}} \approx 1.9168
\end{equation}
and by Theorem \ref{T:2}, or \eqref{E:real2}, we have $V(0) < C \, M^{7/4} E^{-1/4}$, which produces the range
\begin{equation}
\label{E:s>1-T2}
\frac{\beta}{\alpha^{1/4}} > \left(\frac{27}{\frac{3^6\cdot 5^2 \pi^2}{7^5} + \frac12} \right)^{1/4} \approx 1.2460.
\end{equation}
See Figure \ref{F:s>1-N=4} for a graph of thresholds in this case. Note that except for the small data theory (global existence and scattering for small in the invariant norm data), no other information about scattering or blow up thresholds is known in the energy-supercritical case.

\begin{figure}
\includegraphics[scale=1]{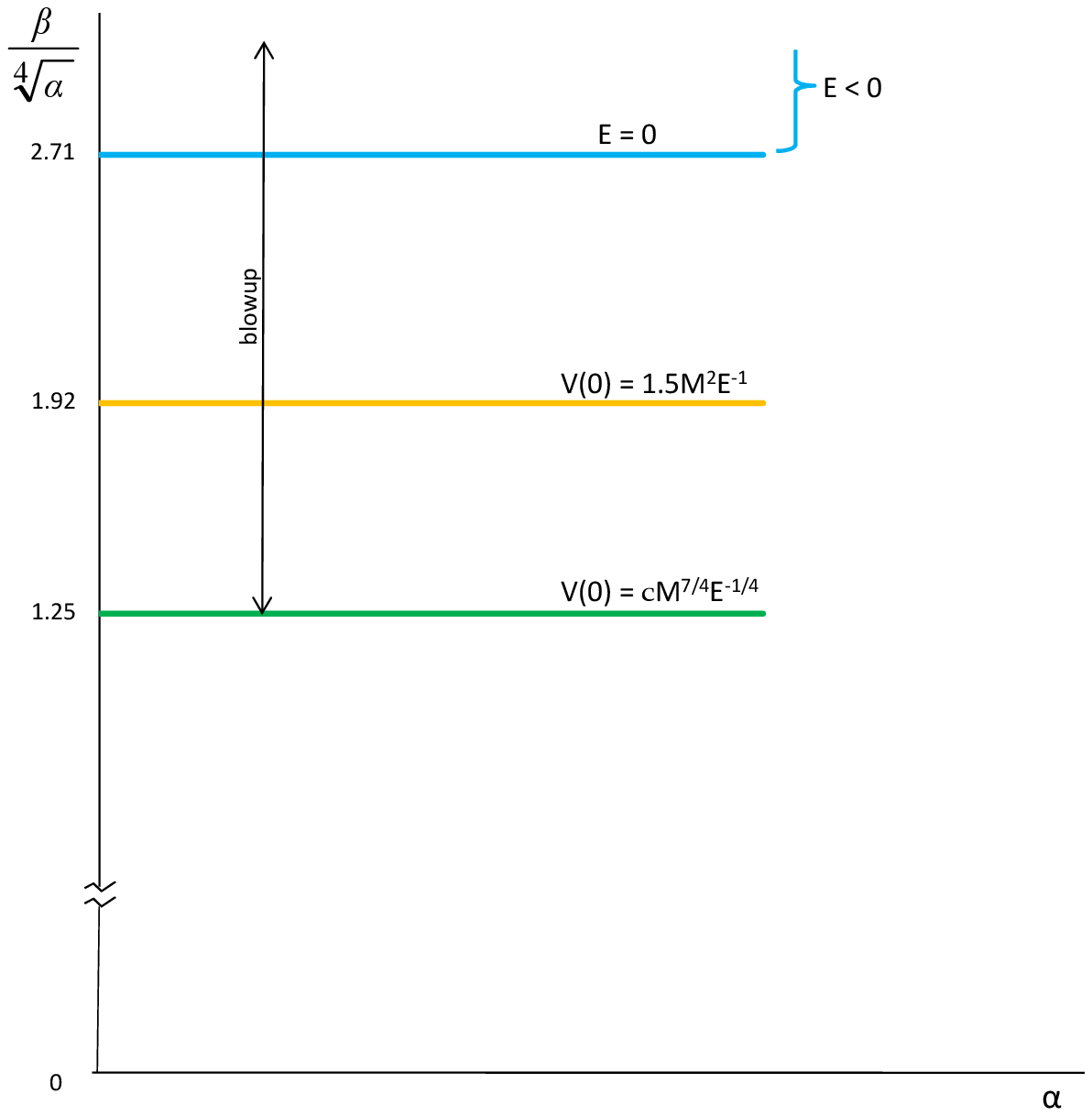}  %
\caption{Global behavior of solutions to the 4d quintic NLS ($s_c=\tfrac32 > 1$) with real Gaussian initial condition $u_0(x) = \beta \, e^{\alpha |x|^2 / 2}$. The top line ($E=0$) indicates the range above which (for $\beta> 2.71 \sqrt[4]\alpha$) solutions blow up by the negative energy criterion. The line denoted by $V(0)= 1.5 M^2 E^{-1}$ is the threshold for blow up from Theorem \ref{T:1}, similarly, the line denoted by $V(0)=c M^{7/4} E^{-1/4}$
is the threshold for blow up from Theorem \ref{T:2}. In this example, Theorem \ref{T:2} gives a wider range of blow up solutions. There are no other thresholds known in this energy-supercritical case.}
 \label{F:s>1-N=4}
\end{figure}

\medskip

\subsubsection{The energy-subcritical case.}
Finally, we consider one-dimensional ($N=1$) example of the NLS equation when $0<s<1$, or the nonlinearity $p>5$.
In this case the scaling index is $s_c = \frac{p-5}{2(p-1)}$, the energy is
\begin{equation}
\label{E:energy-subcrit}
E[u_g]  = \beta^2 \sqrt{\pi \, \alpha} \left(\frac{1}4
- \sqrt{\frac2{(p+1)^3}} \, \frac{\beta^{p-1}}{\alpha} \right).
\end{equation}
Theorem \ref{T:1} guarantees blow up in finite time if
\begin{equation}
\label{E:Thm1-realgaussian}
\frac{\beta^{p-1}}{\alpha} > \frac{p+1}{p-1} \left( \frac{p+1}2 \right)^{\frac{1}2} =: (\kappa_{T1})^{p-1},
\end{equation}
and by Theorem \ref{T:2} the blow-up occurs if
\begin{equation}
\label{E:Thm2-realgaussian}
\frac{\beta^{p-1}}{\alpha} >  \frac{ \frac{p+1}{p-1}  \left( \frac{p+1}2 \right)^{\frac{1}{2}} \frac{1}4}
{\left(\frac{s_c}2 \left(\frac{p+1}2\right)^{\frac{1}2} \left({2\pi}\right)^{\frac{(p-1)}{4}}
(C_{p,1})^{-\left(\frac{(p-1)}{2}+(p+1)\right)}  + \frac1{p-1} \right)} = :(\kappa_{T2})^{p-1}.
\end{equation}
\medskip

\underline{Example $s_c<1$: $p=7$, $N=1$}.
\smallskip

\noindent In this case the scaling index is $s_c = \frac16$, the threshold values from \eqref{E:Thm1-realgaussian} and \eqref{E:Thm2-realgaussian} are $\kappa_{T1} = 1.1776$ and $\kappa_{T2} = 1.1996$. The {\it mass-energy} threshold from Theorem \ref{T:DHR} gives the blow up range when $\beta > 1.2312 \alpha^{1/6}$ and the scattering range $\beta < 1.0844 \alpha^{1/6}$. We graph these results in Figure \ref{F:gauss-nophase-7}.

\begin{figure}
\includegraphics[scale=0.9]{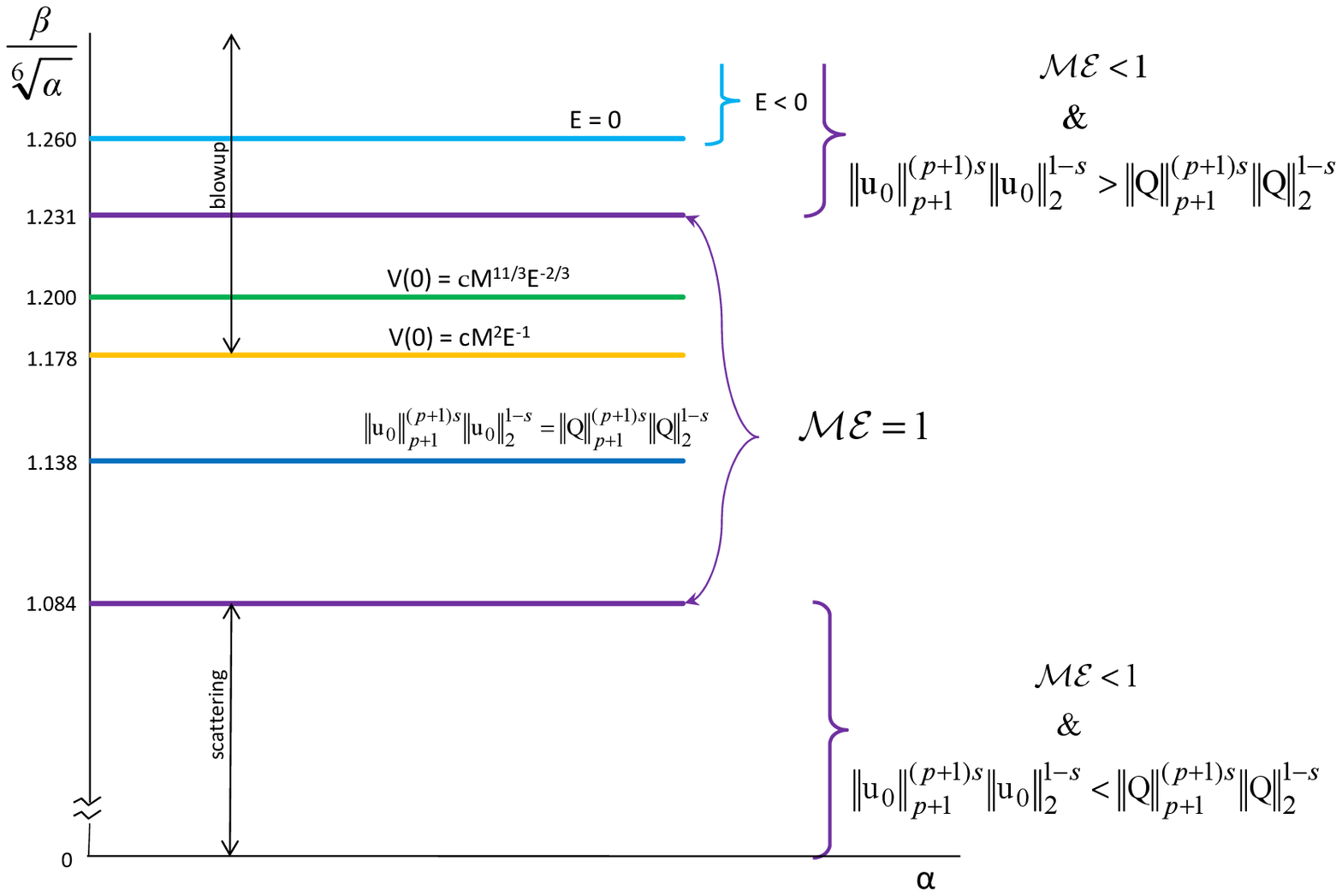}
\caption{Global behavior of solutions to 1d septic ($p=7$) NLS equation ($s_c = \frac16<1$) with the gaussian initial data $u_0(x) = \beta \, e^{\alpha x^2 / 2}$. The top line ($E=0$) indicates the range above which (for $\beta> 1.260 \alpha^{1/6}$) solutions blow up by the negative energy criterion. The line denoted by $V(0)=c M^2 E^{-1}$ is the threshold for blow up from Theorem \ref{T:1}, see \eqref{E:Thm1-realgaussian}; similarly, the line denoted by $V(0)=c M^{11/3} E^{-2/3}$ is the threshold for blow up from Theorem \ref{T:2}, see \eqref{E:Thm2-realgaussian}. In this example, Theorem \ref{T:2} gives a wider range.
The range for {scattering} is given by Theorem \ref{T:DHR}. }
 \label{F:gauss-nophase-7}
\end{figure}

\appendix

\section{Values of sharp constant $C_{p,N}$ }
\label{A:sharp-constants}

\begin{remark}
For $N=1$, we obtain
\begin{equation}
\label{E:C-anyp-N=1}
\left(C_{p,1} \right)^{\frac{3p+1}2} = \frac{(3p+1)^{\frac{p+3}4}}{2(p+1)} \left(\frac{\pi}{(p-1)}\right)^{\frac{p-1}4} \left(\frac{\Gamma\left(\frac{p+1}{p-1}\right)}{\Gamma\left(\frac{p+1}{p-1} + \frac12\right)} \right)^{\frac{p-1}2}.
\end{equation}
For $N=2$, we get
\begin{equation}
\label{E:C-N=2}
\left(C_{p,2}\right)^{2p} = \pi^{\frac{p-1}2} \left(\frac{2p}{p+1}\right)^{\frac{p+1}2}.
\end{equation}
For $N=3$, we have
\begin{equation}
\label{E:C-anyp-N=3}
\left( C_{p,3} \right)^{\frac{5p-1}2} =
\left(\frac{\pi}{3(p-1)}\right)^{\frac{3(p-1)}4}
\frac{(5p-1)^{\frac{3p+1}4}}{2(p+1)}
\left(\frac{\Gamma\left(\frac{p+1}{p-1}\right)}{\Gamma\left(\frac{p+1}{p-1} + \frac32\right)}\right)^{\frac{p-1}2}.
\end{equation}
When $N=4$, we obtain
\begin{equation}
\label{E:C-N=4}
\left(C_{p,4}\right)^{3p-1} = \left(\frac{\pi}{2^{3/2}} \right)^{p-1} \frac{(3p-1)^p}{p^{\frac{p-1}2}(p+1)^{\frac{p+1}2}}.
\end{equation}
Now we fix $p$ and vary the dimension. Let $p=3$. Then
\begin{equation}
\label{E:C-anyN-p=3}
\left(C_{3,N} \right)^{N+4} =  \frac{\pi^{N/2}}{\Gamma(\frac{N}2)}
\frac1{N+2} \left( \frac{N+4}{N} \right)^{\frac{N+2}2}.
\end{equation}
Let $p=5$. Then
\begin{equation}
\label{E:C-anyN-p=5}
\left(C_{5,N} \right)^{2N+6} =  \frac{N}{12} \left(\pi \left(\frac{N+3}{N}\right)\right)^{N+1}
\left[\Gamma\Big(\frac{N+3}2\Big)\right]^{-2}.
\end{equation}
Finally, for $p=7$ we have
\begin{equation}
\label{E:C-anyN-p=7}
\left(C_{7,N} \right)^{3N+8} =
\left(\pi \left(\frac{3N+8}{3N}\right)\right)^{\frac{3N}2} \left(\frac38N+1\right)
\left(\frac{\Gamma(\frac43)}{\Gamma(\frac43+\frac{N}2)}\right)^{3}.
\end{equation}

\end{remark}

\end{document}